\documentclass[11pt]{amsart}
\usepackage{amsfonts}
\usepackage{amscd}
\usepackage{pifont}
\usepackage{epsfig}
\usepackage{amsmath,amsfonts}
\usepackage{graphicx, color}

\newtheorem{thm}{Theorem}
\newtheorem{lemma}{Lemma}[section]
\newtheorem{cor}{Corollary}
\newtheorem{rk}{Remark}[section]

\newcommand\be{\begin{equation}}
\newcommand\ee{\end{equation}}
\newcommand\bea{\begin{eqnarray}}
\newcommand\eea{\end{eqnarray}}
\newcommand\beaa{\begin{eqnarray*}}
\newcommand\eeaa{\end{eqnarray*}}
\newcommand\bay{\begin{array}}
\newcommand\eay{\end{array}}
\newcommand\ba{\begin{align}}
\newcommand\ea{\end{align}}

\newcommand{\R}{\mathbb{R}}

\newcommand\dps{\displaystyle}

\begin{document}

\title[Spreading with two speeds and mass segregation]
{Spreading with two speeds and mass segregation in a diffusive competition system with  free boundaries}
\author{Yihong Du}
\address{Y. Du: School of Science and Technology, University of New England, Armidale, NSW 2351, Australia.}
\email{ydu@une.edu.au}

\author{Chang-Hong Wu}
\address{C-H. Wu: Department of Applied Mathematics, National University of Tainan, Tainan 700, Taiwan}
\email{changhong@mail.nutn.edu.tw}

\thanks{Date: \today.}

\thanks{{\em 2010 Mathematics Subject Classification.} 35K51, 35R35, 92B05.}

\thanks{{\em Key words and phrases: free boundary problem, Lotka-Volterra competition model, asymptotic spreading speed.}}

\begin{abstract}
We investigate the spreading behavior of two invasive species modeled by a
Lotka-Volterra diffusive competition system with two free boundaries in a  spherically symmetric setting.
We show that, for the weak-strong competition case, under suitable assumptions, both species in the system can successfully spread
into the available environment, but their spreading speeds are different, and their population masses tend to segregate, with
 the slower spreading competitor having its population concentrating on an expanding ball, say $B_t$, and the faster spreading competitor
 concentrating on a spherical shell outside $B_t$ that disappears to infinity as time goes to infinity.
\end{abstract}

\maketitle
\setlength{\baselineskip}{18pt}

\section{Introduction}
\setcounter{equation}{0}
In this paper, we investigate the  spreading behavior of two competing species
described by the following free boundary problem in $\mathbb{R}^N$ ($N\geq1$) with spherical symmetry:
\beaa
{\bf(P)}\hspace{1cm}\left\{
 \begin{array}{ll}
 u_t=d\Delta u+ru(1-u-kv) \mbox{ for } 0<r<s_1(t),\ t>0,\\
 v_t=\Delta v+v(1-v-hu) \mbox{ for } 0<r<s_2(t),\ t>0,\\
 u_r(0,t)=v_r(0,t)=0 \mbox{ for } t>0,\\
 u\equiv0\;\mbox{for all}\  r\geq s_1(t)\ \mbox{and}\ t>0,\;   v\equiv0\; \mbox{for all}\   r\geq s_2(t)\ \mbox{and}\ t>0,\\
 s_{1}'(t)=-\mu_1 u_r(s_1(t),t) \mbox{ for } t>0,\;  s_{2}'(t)=-\mu_2 v_r(s_2(t),t) \mbox{ for } t>0,\\
 s_1(0)=s_1^0,\ s_2(0)=s_2^0,\ u(r,0)=u_0(r),\ v(r,0)=v_0(r)\ \mbox{for}\ r\in[0,\infty),
  \end{array}
\right.
\eeaa
where $u(r,t)$ and $v(r,t)$ represent the population densities of  the two competing species at spatial location
$r\, (=|x|)$ and time $t$; $\Delta \varphi:=\varphi_{rr}+\frac{(N-1)}{r}\phi_r$  is the usual Laplace operator acting on spherically symmetric functions.
All the parameters are assumed to be positive, and without loss of generality, we have used a simplified version of the  Lotka-Volterra competition model, which can be obtained from the general model  by a standard change of variables procedure (see, for example, \cite{DWZ}).
The initial data $(u_0,v_0,s_1^0,s_2^0)$ satisfies
\bea\label{ic}
\left\{
 \begin{array}{ll}
s_1^0>0,\ s_2^0>0,\ u_0\in C^2([0,s_1^0]),\ v_0\in C^2([0,s_2^0]),\ u_0'(0)=v_0'(0)=0,\\
u_0(r)>0\ \mbox{for}\ r\in[0,s_1^0),\ u_0(r)=0\ \mbox{for}\ r\geq s_1^0,\\
v_0(r)>0\ \mbox{for}\ r\in[0,s_2^0),\ v_0(r)=0\ \mbox{for}\ r\geq s_2^0.
 \end{array}
\right.
\eea
In this model, both species invade the environment through their own free boundaries: the species $u$ has a spreading front at $r=s_1(t)$, while $v$'s spreading front is at $r=s_2(t)$. For the mathematical treatment, we have extended $u(r,t)$ from its population range $r\in [0, s_1(t)]$
to $r>s_1(t)$ by 0, and
extended $v(r,t)$ from $r\in [0, s_2(t)]$ to $r>s_2(t)$ by 0.

{The global existence and uniqueness of the solution to problem {\bf(P)} under \eqref{ic} can be established
by  the approach in \cite{GW2} with suitable changes. In fact, the local existence and uniqueness proof can cover a rather general class of such free boundary systems. The assumption in {\bf (P)} that $u$ and $v$ have independent free boundaries causes some difficulties but this can be handled by following the approach in \cite{GW2} with suitable modifications and corrections. The details are given in the Appendix at the end of the paper.

We say $(u,v,s_1, s_2)$ is a (global classical) solution of {\bf (P)}  if
\beaa
(u,v,s_1,s_2)\in C^{2,1}(D^1)\times C^{2,1}(D^2)\times C^1([0,+\infty))\times C^1([0,+\infty)),
\eeaa
where
\[
D^1:=\{(r,t): r\in [0, s_1(t)], \; t>0\},\;\; D^2:=\{(r,t): r\in [0, s_2(t)], \; t>0\},
\]
and all the equations in {\bf (P)} are satisfied pointwisely.
By the Hopf boundary lemma, it is easily seen that, for $i=1,2$ and $t>0$, $s_i'(t)>0$. Hence
\[
\mbox{ $s_{i,\infty}:=\lim_{t\to\infty} s_i(t)$ }
\]
 is well-defined.

We are interested in the long-time behavior of {\bf(P)}. In order to gain a good understanding, we focus on some interesting special cases.
Our first assumption is that
\begin{equation}
\label{h-k}
0<k<1<h.
\end{equation}
It is well known that under this assumption, when restricted over a fixed bounded domain $\Omega$ with no-flux boundary conditions,
the unique solution $(\tilde u(x,t), \tilde v(x,t))$ of the corresponding problem of {\bf (P)} converges to $(1,0)$ as $t\to\infty$ uniformly for $x\in\overline\Omega$. So in the long run, the species $u$ drives $v$ to extinction and wins the competition. For this reason,
condition \eqref{h-k} is often referred to as the case that  $u$ is  superior  and $v$ is  inferior  in the competition.
This is often referred to as a weak-strong competition case. A symmetric situation is $0<h<1<k$.

The case $h,k\in (0,1)$ is called the weak competition case (see \cite{Wu2}), while the case $h,k\in (1,+\infty)$ is known as the strong competition case. In these cases, rather different long-time dynamical behaviors are expected.

In this paper, we will focus on problem {\bf (P)} for the weak-strong competition case \eqref{h-k}, and demonstrate
a rather interesting phenomenon, where $u$ and $v$ both survive in the competition, but they
spread into the new territory with different speeds, and their population masses tend to segregate,
with the population mass  of $v$ shifting to infinity as $t\to\infty$.

For {\bf (P)} with space dimension $N=1$, such a phenomenon was discussed in \cite{GW2}, though less precisely than here.
It is shown in Theorem 5 of \cite{GW2}
that under \eqref{h-k} and some additional conditions, both species can spread successfully, in the sense that
\begin{itemize}
\item[(i)] $s_{1,\infty}=s_{2,\infty}=\infty$,
\item[(ii)]
there exists $\delta>0$ such that for all $t>0$,
\[
\mbox{$u(x,t)\geq \delta $ for $x\in I_u(t)$,\; $v(x,t)\geq \delta$ for $x\in I_v(t)$,}
\]
where $I_u(t)$ and $I_v(t)$ are intervals of length at least $\delta$ that vary continuously in $t$.
\end{itemize}
At the end of the paper \cite{GW2}, the question of determining the spreading speeds for both species was raised as an open problem for future investigation.

In this paper, we will determine, for such a case,
$
\lim_{t\to\infty}\frac{s_i(t)}{t},\; i=1,2$; so in particular, the open problem of \cite{GW2} on the spreading speeds is resolved here.
Moreover, we also obtain a much better understanding of the long-time behavior of $u(\cdot, t)$ and $v(\cdot, t)$,
for all dimensions $N\geq 1$. See Theorem 1 and Corollary 1 below for details.

A crucial new ingredient (namely $c^*_{\mu_1}$ below) in our approach here comes from recent research on another closely related problem
 (proposed in \cite{DL2}):
\beaa{\bf(Q)}\hspace{1.5cm}
\begin{cases}
u_t=d \Delta u+ru(1-u-kv),\quad 0\leq r<h(t),\ t>0,\\
v_t=\Delta v+v(1-v-hu),\quad 0\leq r<\infty,\ t>0,\\
u_r(0,t)=v_r(0,t)=0,\ u(r,t)=0,\quad h(t)\leq r<\infty,\ t>0,\\
h'(t)=-\mu_1 u_r(h(t),t),\quad t>0,\\
h(0)=h_0, u(r,0)=\hat{u}_0(r),\quad  0\leq r\leq h_0,\\
v(r,0)=\hat{v}_0(r),\quad 0\leq r<\infty,
\end{cases}
\eeaa
where $0<k<1<h$ and
\bea\label{ic-Q}
\left\{
 \begin{array}{ll}
\hat{u}_0\in C^2([0,h_0]),\ \hat{u}_0'(0)=\hat{u}_0(h_0)=0,\ \hat{u}_0>0\ \mbox{in $[0,h_0)$},\\
\hat{v}_0\in C^2([0,\infty))\cap L^{\infty}((0,\infty)),\ \hat{u}_0'(0)=0,\ \hat{v}_0\geq(\not\equiv)0\ \mbox{in $[0,\infty)$}.
 \end{array}
\right.
\eea
In problem  {\bf(Q)} the inferior competitor $v$ is assumed to be a native species already established  in the environment,
while the superior competitor $u$ is invading the environment via the free boundary $r=h(t)$. Theorem 4.3 in \cite{DL2} gives a
spreading-vanishing dichotomy for {\bf (Q)}: Either
\begin{itemize}
\item {(\it Spreading of $u$}) $\lim_{t\to\infty} h(t)=\infty$ and
\[
\lim_{t\to\infty} (u(r, t), v(r,t))=(1,0) \mbox{ locally uniformly for $r\in [0,\infty)$, or}
\]
\item {(\it Vanishing of $u$}) $\lim_{t\to\infty} h(t)<\infty$ and
\[
\lim_{t\to\infty} (u(r, t), v(r,t))=(0,1) \mbox{ locally uniformly for $r\in [0,\infty)$.}
\]
\end{itemize}
Sharp criteria for spreading and vanishing of $u$ are also given in \cite{DL2}. When spreading of $u$ happens,
an interesting question is whether there exists an asymptotic spreading speed, namely whether $\lim_{t\to\infty}\frac{h(t)}{t}$ exists.
This kind of questions, similar to the one being asked in \cite{GW2} mentioned above, turns out to be rather difficult to answer for systems of equations with free boundaries.
Recently, Du, Wang and Zhou \cite{DWZ} successfully established the spreading speed for {\bf (Q)}, by making use of
 the following so-called semi-wave system:
\bea\label{semi-wave system}
\begin{cases}
cU'+d U''+rU(1-U-kV)=0,\quad -\infty<\xi<0,\\
cV'+V''+v(1-V-hU)=0,\quad -\infty<\xi<\infty,\\
U(-\infty)=1,\quad U(0)=0,\quad U'(\xi)<0 =U(-\xi),\ \xi<0,\\
V(-\infty)=0,\quad V(+\infty)=1,\quad V'(\xi)>0,\ \xi\in\R.
\end{cases}
\eea
It was shown that \eqref{semi-wave system} has a unique solution if $c\in[0, c_0)$, and it has no solution if $c\geq c_0$,
where
\[
c_0\in[\,2\sqrt{rd(1-k)},2\sqrt{rd}\,]
\]
 is the minimal speed for the traveling wave solution studied in \cite{Kan-on}.
More precisely,  the following result holds:
\smallskip

\noindent
{\bf Theorem A.} (Theorem\, 1.3 of \cite{DWZ}) {\it
Assume that $0<k<1<h$. Then for each $c\in[0,c_0)$, \eqref{semi-wave system} has a unique solution
$(U_c,V_c)\in [C(\R)\cap C^2([0,\infty))]\times C^2(\R)$, and it has no solution for $c\geq c_0$. Moreover,
\begin{itemize}
\item[(i)] if $0\leq c_1< c_2<c_0$, then
\beaa
U'_{c_1}(0)<U'_{c_2}(0),\; U_{c_1}(\xi)>U_{c_2}(\xi) \mbox{ for } \xi<0,\; V_{c_2}(\xi)>V_{c_1}(\xi) \mbox{ for } \xi\in\R;
\eeaa
\item[(ii)] the mapping $c\mapsto(U_c,V_c)$ is continuous from $[0,c_0)$ to $C^2_{loc}((-\infty,0])\times C^2_{loc}(\R)$ with
\beaa
\lim_{c\to c_0}(U_c,V_c)=(0,1)\quad\mbox{in $C^2_{loc}((-\infty,0])\times C^2_{loc}(\R)$};
\eeaa
\item[(iii)] for each $\mu_1>0$, there exists a unique $c=c^*_{\mu_1}\in(0,c_0)$ such that
\beaa
\mu_1 U'_{c^*_{\mu_1}}(0)=c^*_{\mu_1}\quad \mbox{and $c^*_{\mu_1}\nearrow c_0$ as ${\mu_1}\nearrow\infty$}.
\eeaa
\end{itemize}
}

The spreading speed for {\bf(Q)} is established as follows.
\smallskip

\noindent
{\bf Theorem B.} (Theorem\, 1.1 of \cite{DWZ}) {\it
Assume that $0<k<1<h$. Let $(u,v,h)$ be the solution of {\bf(Q)} with \eqref{ic-Q} and
\bea\label{inf v0}
\liminf_{r\to\infty} \hat{v}_0(r)>0.
\eea
If $h_{\infty}:=\lim_{t\to\infty}h(t)=\infty$,
then
\beaa
\lim_{t\to\infty}\frac{h(t)}{t}=c^*_{\mu_1},
\eeaa
where $c^*_{\mu_1}$ is given in Theorem~A.
}

It turns out that $c^*_{\mu_1}$ also plays an important role in determining the long-time dynamics of {\bf (P)}.
In order to describe the second crucial number for the dynamics of {\bf (P)} (namely $s^*_{\mu_2}$ below),
let us recall that, in the absence of the species $u$, problem {\bf(P)} reduces to a single species  model
 studied by Du and Guo \cite{DG}, who
generalized the model proposed by Du and Lin \cite{DL} from one dimensional space to high dimensional space with spherical symmetry.
In such a case, a spreading-vanishing dichotomy holds for $v$, and when spreading happens, the spreading speed of $v$
is 
related to the following problem
\bea\label{semi-wave eq}
\left\{
 \begin{array}{ll}
 dq''+sq'+q(a-bq)=0\quad \mbox{in $(-\infty,0)$},\\
 q(0)=0,\quad q(-\infty)=a/b,\quad q(\xi)>0\quad \mbox{in $(-\infty,0)$}.
  \end{array}
\right.
\eea
More precisely, by Proposition 2.1 in \cite{BDK} (see also  Proposition 1.8 and Theorem 6.2 of \cite{DLou}), the following result holds:
\smallskip

\noindent
{\bf Theorem C.} {\it
For fixed {$a,b,d,\mu_2>0$}, there exists a unique
{$s=s^*(a,b,d,\mu_2)\in (0, 2\sqrt{ad})$}
and
a unique solution $q^*$ to \eqref{semi-wave eq} with {$s=s^*(a,b,d,\mu_2)$} 
such that 
{$(q^*)'(0)=-s^*(a,b,d,\mu_2)/\mu_2$}.
Moreover, $(q^*)'(\xi)<0$ for all $\xi\leq0$.}

\bigskip

Hereafter, we shall denote $s^*_{\mu_2}:=s^*(1,1,1,\mu_2)$. It turns out that the long-time behavior of {\bf (P)} depends crucially on whether $c^*_{\mu_1}<s^*_{\mu_2}$ or
$c^*_{\mu_1}>s^*_{\mu_2}$.  As demonstrated in Theorems 1 and 2 below, in the former case, it is possible for both species to spread successfully, while in the latter case, at least one species has to vanish eventually.

Let us note that while the existence and uniqueness of $s^*_{\mu_2}$ is relatively easy to establish (and has been  used in \cite{GW2} and other papers to estimate the spreading speeds for various systems), this is not the case for $c^*_{\mu_1}$, which takes more than half of the length of \cite{DWZ} to establish. The main advance of this research from \cite{GW2} is achieved by making use of $c^*_{\mu_1}$.

\begin{thm}\label{thm1}
 Suppose \eqref{h-k} holds and
\begin{equation}
\label{A2}
c^*_{\mu_1}<s^*_{\mu_2}.
\end{equation}
Then one can choose initial functions $u_0$ and $v_0$ properly such that
 the unique solution $(u,v, s_1, s_2)$ of
{\bf (P)} satisfies
\[
\lim_{t\to\infty} \frac{s_1(t)}{t}=c_{\mu_1}^*,\;\; \lim_{t\to\infty}\frac {s_2(t)}{t}=s^*_{\mu_2},
\]
and  for every small $\epsilon>0$,
\begin{equation}\label{u-v-1}
\lim_{t\to\infty} (u(r, t), v(r,t))=(1, 0) \mbox{ uniformly for $r\in [0,(c_{\mu_1}^*-\epsilon)t]$},
\end{equation}
\begin{equation}\label{v-2}
\lim_{t\to\infty}v(r,t)=1 \mbox{ uniformly for $r\in [(c^*_{\mu_1}+\epsilon)t, (s^*_{\mu_2}-\epsilon)t]$}.
\end{equation}
\end{thm}

Before giving some explanations regarding the condition \eqref{A2} and the choices of $u_0$ and $v_0$ in the above theorem, let us first note that the above conclusions indicate that
the $u$ species spread at the asymptotic speed $c^*_{\mu_1}$, while $v$ spreads at the {\it faster} asymptotic speed $s^*_{\mu_2}$. Moreover, \eqref{u-v-1} and \eqref{v-2}
imply that the population mass of $u$ roughly concentrates on the expanding ball $\{r<c^*_{\mu_1}t\}$, while that of $v$
concentrates on the expanding spherical shell  $\{c^*_{\mu_1}t<r< s^*_{\mu_2}t\}$ which shifts to infinity
as $t\to\infty$. We also note that, apart from a relatively thin  coexistence
shell around $r=c^*_{\mu_1}t$,  the population masses of $u$ and $v$ are largely segregated for all large time.
Clearly this gives a more precise description for the spreadings of $u$ and $v$ than that in Theorem 5 of \cite{GW2} (for $N=1$) mentioned above.

We now look at some simple sufficient conditions for \eqref{A2}. We note that
$c^*_{\mu_1}$
is independent of $\mu_2$ and the initial functions.
From the proof of Lemma 2.9 in \cite{DWZ}, we see that
$c^*_{\mu_1}\to 0$ as $\mu_1\to 0$. Therefore when all the other parameters are fixed,
\[
\mbox{\eqref{A2} holds for all small $\mu_1>0$.}
\]
A second sufficient condition can be found by using Theorem A (iii), which implies $c^*_{\mu_1}< c_0\leq 2\sqrt{rd}$ for all $\mu_1>0$.
It follows that
\[
\mbox{\eqref{A2}  holds for all  $\mu_1>0$ provided that $2\sqrt{rd}\leq s^*_{\mu_2}$.
}
\]
{Note that $2\sqrt{rd}\leq s^*_{\mu_2}$ holds if $\sqrt{rd}<1$ and $\mu_2\gg1$ since $s^*_{\mu_2}\to 2$ as $\mu_2\to\infty$.}

\medskip

For the conditions in Theorem~\ref{thm1} on the initial functions $u_0$ and $v_0$, the simplest ones are
given in the corollary below.
\begin{cor}\label{cor1}
Assume \eqref{h-k} and \eqref{A2}. Then there exists a large positive constant $C_0$ depending on $s^0_1$ such that the conclusions of Theorem~\ref{thm1} hold if
\begin{itemize}
\item[(i)] $\|u_0\|_{L^{\infty}([0,s^0_1])}\leq1$ with $s^0_1\geq R^*\sqrt{d/[r(1-k)]}$,
\item[(ii)]  for some $x_0\geq C_0$ and $ L\geq C_0$, $v_0(r)\geq1$ for $r\in[x_0,x_0+L]$.
\end{itemize}
Here $R^*$ is uniquely determined by $\lambda_1(R^*)=1$, where $\lambda_1(R)$ is the principal eigenvalue of
\beaa
-\Delta \phi= \lambda \phi\quad \mbox{in $B_R$},\quad \phi=0\quad \mbox{on $\partial B_R$}.
\eeaa
\end{cor}

Roughly speaking, conditions (i) and (ii) above (together with \eqref{h-k} and \eqref{A2}) guarantee that $u$ does not vanish yet it cannot spread too fast initially,
and the initial population of $v$ is relatively well-established in some part of the environment where  $u$ is absent, so with its fast spreading speed $v$ can outrun the superior but slower competitor $u$.
In Section 2, weaker sufficient conditions on $u_0$ and $v_0$ will be given (see  {\bf (B1)} and  {\bf (B2)} there).

\medskip

Next we describe the long-time behavior of {\bf (P)} for the case
\begin{equation}\label{A3}
c^*_{\mu_1}>s^*_{\mu_2}.
\end{equation}
We will show that, in this case, no matter how the initial functions $u_0$ and $v_0$ are chosen, at least one of $u$ and $v$ will vanish eventually. As in \cite{GW2}, we say $u$ (respectively $v$) {\it vanishes eventually} if
\[
s_{1,\infty}<+\infty \mbox{ and } \lim_{t\to+\infty}\|u(\cdot, t)\|_{L^\infty([0, s_1(t)])}=0
\]
\[ (\mbox{respectively, }\;
s_{2,\infty}<+\infty \mbox{ and } \lim_{t\to+\infty}\|v(\cdot, t)\|_{L^\infty([0, s_1(t)])}=0);
\]
and we say $u$ (respectively $v$) {\it spreads successfully} if
\[\mbox{
$s_{1,\infty}=\infty$ and
there exists $\delta>0$ such that,}
\]
\[
\mbox{$u(x,t)\geq \delta $ for $x\in I_u(t)$ and $t>0$,}
\]
where $I_u(t)$ is an interval of length at least $\delta$ that varies continuously in $t$
(respectively,
\[\mbox{$s_{2,\infty}=\infty$ and
there exists $\delta>0$ such that}
\]
\[\mbox{
$v(x,t)\geq \delta$ for $x\in I_v(t)$ and $t>0$,}
\]
where  $I_v(t)$ is an interval of length at least $\delta$ that varies continuously in $t$).

\medskip

For fixed $d,r,h,k>0$ satisfying \eqref{h-k}, we define
\beaa
\mathcal{B}=\mathcal{B}(d,r,h,k):=\Big\{(\mu_1,\mu_2)\in \mathbb{R}_+\times\mathbb{R}_+: \ c^*_{\mu_1}>s^*_{\mu_2}\Big \}.
\eeaa
Note that $\mathcal{B}\neq\emptyset$ since $s^*_{\mu_2}\to 0$ as $\mu_2\to 0$ and $c^*_{\mu_1}>0$ is independent of $\mu_2$.

We have the following   result.
\begin{thm}\label{prop-trichotomy} Assume that \eqref{h-k} holds.  If $(\mu_1,\mu_2)\in \mathcal{B}$, then at least one of the species $u$ and $v$ vanishes eventually. More precisely, depending on the choice of $u_0$ and $v_0$, exactly one of the following occurs for the unique solution $(u,v,s_1,s_2)$ of  {\bf(P)}$:$
\begin{itemize}
\item[(i)] Both  species $u$ and $v$ vanish eventually.
\item[(ii)] The species $u$ vanishes eventually and $v$ spreads successfully.
\item[(iii)] The species $u$ spreads successfully and $v$ vanishes eventually.
\end{itemize}
\end{thm}

Note that $(\mu_1,\mu_2)\in \mathcal{B}$ if and only if \eqref{A3} holds.
Theorem~\ref{prop-trichotomy} can be proved along the lines of the proof of \cite[Corollary 1]{GW2} with some suitable changes. When $N=1$, Theorem 2 slightly improves the conclusion of Corollary 1 in \cite{GW2}, since it is easily seen that $\mathcal{A}\subset \mathcal{B}$ (due to
$s^*(r(1-k),r,d,\mu_1)\leq c^*_{\mu_1}$), where $\mathcal{A}:=\Big\{(\mu_1,\mu_2)\in \mathbb{R}_+\times\mathbb{R}_+: \ s^*(r(1-k),r,d,\mu_1)>s^*_{\mu_2}\Big \}$ is given in \cite{GW2}.

\begin{rk}{\rm
We note that by suitably choosing the initial functions $u_0$ and $v_0$ and the parameters $\mu_1$ and $\mu_2$, all the three possibilities  in Theorem~\ref{prop-trichotomy} can occur. For example, for given $u_0$ and $v_0$
with $s_1^0<R^*\sqrt{\frac{d}{r}}$ and $s_2^0< R^*$, then scenario (i) occurs
as long as both $\mu_1$ and $\mu_2$ are small enough and $(\mu_1, \mu_2)\in\mathcal{B}$ (which  can be proved by using the argument in \cite[Lemma 3.8]{DL}). If next we modify $v_0$ such that $s_2^0\geq R^*$,
then $u$ still vanishes eventually but $v$ will spread successfully, which leads to scenario (ii). For scenario (iii) to occur, we can take $s_1^0\geq R^*\sqrt{\frac{d}{r(1-k)}}$
and $\mu_2$ small enough.}
\end{rk}

Our results here suggest that in the weak-strong competition case, co-existence of the two species over a common (either moving or stationary) spatial region  can hardly happen. This contrasts sharply to the weak competition case
($h,k\in(0,1)$), where coexistence often occurs; see, for example \cite{Wu2, WND}.

Before ending this section, we mention some further references that form part of the background of this research.
Since the work \cite{DL}, there have been tremendous
efforts towards developing analytical tools to deal with more general single species models with free boundaries; see
 \cite{BDK,DG2,DGP,DLiang,DLou,DLZ,DMZ,DMZ2,Kaneko,KY,KY2,LLZ,MW,PZ,Wang1,ZX} and references therein.
Related works for two species models can be found in, for example, \cite{DL2,DWZ,GW1,Wang0,Wang2,WZ,WZ1,Wu1,Wu2,ZW}.
The issue of the spreading speed for single species models in homogeneous environment  has been well studied, and we refer to \cite{DMZ,DMZ2} for some sharp estimates.
Some of the theory on single species models can be used to estimate the spreading speed for two species models;
however, generally speaking, only rough upper and lower bounds can be obtained via this approach.

The rest of this paper is organized as follows.
In Section 2, we shall  prove our main result, Theorem \ref{thm1}, based on the comparison principle and on
the  construct of various auxiliary functions as comparison solutions to {\bf(P)}. Section 3 is an appendix, where we prove the local
{and global}
existence and uniqueness of solutions to a wide class of problems including {\bf (P)} as a special case, and we also sketch the proof of Theorem 2.

\section{Proof of Theorem \ref{thm1}}
\setcounter{equation}{0}

We start by establishing several technical lemmas.

\begin{lemma}\label{lem:q}
Let $\mu_2>0$ and $s^*_{\mu_2}$ be given in Theorem C.
Then for each $s\in(0,s^*_{\mu_2})$, there exists a unique $z=z(s)>0$
such that the solution $q_s$ of the initial value problem
\beaa\label{bdry problem}
\left\{
 \begin{array}{ll}
 q''+sq'+q(1-q)=0\quad \mbox{in $(-\infty,0)$},\\
 q(0)=0,\quad q'(0)=-s^*_{\mu_2}/{\mu_2}
  \end{array}
\right.
\eeaa
satisfies
 $q_s'(-z(s))=0$ and $q_s'(z)<0$ for $z\in (-z(s),0)$. Moreover, $q_s(-z(s))$ is continuous in $s$ and
\beaa
z(s)\nearrow\infty,\quad q_s(-z(s))\nearrow1\quad \mbox{as $s\nearrow s^*_{\mu_2}$}.
\eeaa
\end{lemma}
\begin{proof}
The conclusions follow directly  from Proposition~2.4 in \cite{KM}.
\end{proof}

\begin{lemma}\label{lem:estimate1}
Let $(u,v,s_1,s_2)$ be a solution of {\bf(P)} with $s_{1,\infty}=s_{2,\infty}=\infty$. Suppose that
\beaa
\limsup_{t\to\infty}\frac{s_1(t)}{t}<c_1<c_2<\liminf_{t\to\infty}\frac{s_2(t)}{t}
\eeaa
for some positive constants $c_1$ and $c_2$.
Then for any $\varepsilon>0$, there exists $T>0$ such that
\bea
&&v(r,t)<1+\varepsilon\quad \mbox{for all $t\geq T$ and $r\in[0,\infty)$},\label{upper bound}\\
&&v(r,t)>1-\varepsilon\quad \mbox{for all $t\geq T$ and $r\in[c_1 t,c_2 t]$}.\label{lower bound}
\eea
\end{lemma}
\begin{proof}
Let $\bar{w}$ be the solution of $w'(t)=w(1-w)$ with initial data $w(0)=\|v_0\|_{L^{\infty}}.$
By the standard comparison principle, $v(x,t)\leq \bar{w}(t)$ for all $t\geq0$.
Since $\bar{w}\to 1$ as $t\to\infty$, there exists $T>0$ such that \eqref{upper bound} holds.

Before proving \eqref{lower bound},
we first show $\limsup_{t\to\infty}s_2(t)/t\leq s^*_{\mu_2}$ by simple comparison.
Indeed, it is easy to check that $(v,s_2)$ forms a subsolution of
\beaa
\left\{
 \begin{array}{ll}
 \bar{w}_t=\Delta\bar{w}+\bar{w}(1-\bar{w}),\ 0<r<\bar{\eta}(t),\ t>0,\\
 \bar{w}_r(0,t)=0,\ \bar{w}(\bar{\eta}(t),t)=0,\ t>0,\\
 \bar{\eta}'(t)=-\mu_2 \bar{w}_r(\bar{\eta}(t),t),\ t>0,\\
 \bar{\eta}(0)=s_2^0,\ \bar{w}(r,0)=v_0(r),\ r\in[0,s_2^0],
 \end{array}
\right.
\eeaa
By the comparison principle (Lemma 2.6 of \cite{DG}), $\bar{\eta}(t)\geq s_{2}(t)$ for all $t$, which implies that $\bar{\eta}(\infty)=\infty$.
It then follows from Corollary~3.7 of \cite{DG} that $\bar{\eta}(t)/t\to s^*_{\mu_2}$ as $t\to\infty$. Consequently, we have
\[
 \limsup_{t\to\infty}\frac{s_2(t)}{t}\leq \lim_{t\to\infty}\frac{\bar{\eta}(t)}{t}=s^{*}_{\mu_2}.
\]
It follows that $c_2<s^{*}_{\mu_2}$.

We now prove \eqref{lower bound} by using a contradiction argument.
Assume that the conclusion does not hold. Then there exist small $\epsilon_0>0$,  $t_k \uparrow \infty$ and $x_k \in [c_1 t_k, c_2 t_k]$ such that
\bea\label{assumption}
v(x_k, t_k)\leq 1-\epsilon_0  \quad \mbox{for all $k\in\mathbb{N}$}.
\eea
Up to passing to a subsequence we may assume that $p_k:=x_k/t_k \to p_0$ for some $p_0 \in [c_1,c_2]$
as $k\to\infty$.

We want to show that
\bea\label{goal-lem22}
\limsup_{k\to\infty} v(x_k, t_k)>1-\epsilon_0,
\eea
which would give the desired contradiction (with \eqref{assumption}).
To do so, we define
$$w_k(R,t)= v(R + p_k t, t).$$
Then $w_k$ satisfies
\beaa
w_t = w_{RR} + \Big[\frac{N-1}{R+p_k t} + p_k\Big] w_R + w ( 1-w)\quad {\rm for}\ -p_k t<R<s_2(t)-p_k t,\ t\geq t_1.
\eeaa

Recall that $0<c_1<c_2<s_{\mu_2}^*<2$, $x_k=p_kt_k$ and $p_k\to p_0\in [c_1, c_2]\subset (0,2)$. Hence there exists large positive $L$ such that for
all $L_1$, $L_2\in [L, \infty)$,
the problem
\begin{equation}\label{z}
z_{RR}+p_0z_R+z(1-z)=0 \mbox{ in } (-L_2, L_1),\; z(-L_2)=z(L_1)=0
\end{equation}
has a unique positive solution $z(R)$ and $z(0)>1-\epsilon_0$.

Fix $L_1\geq L$, $p\in (p_0, 2)$ and define
\[
\phi(R)=e^{-\frac p2 R}\cos\frac{\pi R}{2L_1}.
\]
It is easily checked that
\[
-\phi''-p\phi'=\left[\frac{p^2}{4}+\frac{\pi^2}{4L_1^2}\right]\phi\quad \mbox{ for } R\in [-L_1, L_1],\; \phi(\pm L_1)=0.
\]
Moreover, there exists a unique  $R_0\in (-L_1, 0)$ such that
\[
\phi'(R)<0 \mbox{ for } R\in (R_0, L_1],\; \phi'(R_0)=0.
\]
We may assume that $L_1$ is large enough such that
\[
\tilde p:=\frac{p^2}{4}+\frac{\pi^2}{4L_1^2}<1.
\]
We then choose $L_2>L$ such that
\[
\mbox{$\tilde L:=L_2+R_0>0$ and}
\;
\frac{\pi^2}{4\tilde L^2}<\tilde p.
\]
Set
\[
\tilde\phi(R):=\left\{\begin{array}{ll}
\phi(R), & R\in [R_0, L_1],\vspace{0.2cm}\\
\phi(R_0)\cos\frac{\pi(R-R_0)}{2\tilde L},& R\in [-L_2, R_0).
\end{array}
\right.
\]
Then clearly
\[
-\tilde\phi''=\frac{\pi^2}{4\tilde L^2}\tilde \phi,\quad \tilde\phi'>0\ \mbox{ in } (-L_2, R_0),\; \tilde \phi(-L_2)=0=\tilde\phi'(R_0).
\]
Since
\[
\frac{N-1}{R+p_kt}+p_k<p
\]
for all large $k$ and large $t$,
we further obtain, for such $k$ and $t$, say $k\geq k_0$ and $t\geq  T_1$,
\begin{equation}
\label{tilde-phi}
-\tilde\phi''-\left[\frac{N-1}{R+p_kt}+p_k\right]\tilde\phi'\leq -\tilde\phi''-p\chi_{[R_0,L_1]}\tilde\phi'\leq \tilde p\,\tilde \phi
\mbox{ for } R\in (-L_2, L_1).
\end{equation}
The above differential inequality should be understood in the weak sense since $\tilde\phi''$ may have a jump at $R=R_0$.

We now fix $T_0\geq T_1$ and observe that
\[
v(R, T_0)>0 \mbox{ for } R\in [0, s_2(T_0)],\; c_2<\liminf_{t\to\infty}\frac{s_2(t)}{t}.
\]
Hence if $T_0$ is large enough then for $R\in [-L_2, L_1]$ and $t\geq T_0$ we have
\[
0<-L_2+c_1t\leq R+p_kt\leq L_1+c_2t<s_2(t) \mbox{ for all } k\geq 1.
\]
It follows that
\[
w_k(R, T_0)=v(R+p_kT_0,T_0)\geq \sigma_0:=\min_{R\in[0, L_1+c_2T_0]}v(R,T_0)>0 \mbox{ for } R\in [-L_2, L_1],\; k\geq 1.
\]

Let $z_k(R,t)$ be the unique solution of
\beaa
z_t =  z_{RR} + \Big[\frac{N-1}{R+p_k t} + p_k\Big]z_R + z(1-z),\quad z(L_1, t) = z(-L_2,t) =0.
\eeaa
with initial condition
\[
z_k(R, T_0)= w_k(R, T_0),\,\; R\in [-L_2, L_1].
\]
The comparison principle yields
\[
w_k(R,t)\geq z_k(R, t) \mbox{ for } R\in [-L_2, L_1],\; t\geq T_0,\; k\geq 1
\]
since
$w_k(R, t)>0=z_k(R,t)$ for $R\in\{-L_2, L_1\}$ and $t> T_0$,\; $k\geq 1$.

On the other hand, if we choose $\delta>0$ sufficiently small, then
$\underline z(R):=\delta\tilde\phi(R)\leq \sigma_0$ for $R\in [-L_2, L_1]$ and due to \eqref{tilde-phi},
$\underline z(R)$ satisfies
\[
-\underline z''-\left[\frac{N-1}{R+p_kt}+p_k\right]\underline z'\leq \underline z(1-\underline z) \mbox{ for } R\in (-L_2, L_1),\; t\geq T_0,\; k\geq k_0.
\]
We thus obtain
\[
z_k(R,t)\geq \underline z(R) \mbox{ for } R\in [-L_2, L_1],\; t\geq T_0,\; k\geq k_0.
\]

We claim that
\begin{equation}\label{z_k}
\lim_{k\to\infty} z_k(0, t_k)=z(0)>1-\epsilon_0,
\end{equation}
where $z(R)$ is the unique positive solution of \eqref{z}, which then gives
\[
\limsup_{k\to\infty}v(x_k,t_k)=\limsup_{k\to\infty}w_k(0,t_k)\geq\limsup_{k\to\infty}z_k(0,t_k)>1-\epsilon_0,
\]
and so \eqref{goal-lem22} holds.

It remains to prove \eqref{z_k}. Set
\[
Z_k(R,t):=z_k(R, t_k+t).
\]
Then $Z_k$ satisfies
\[
(Z_k)_t=(Z_k)_{RR}+\left[ \frac{N-1}{R+p_k(t_k+t)}+p_k\right](Z_k)_R+Z_k(1-Z_k) \mbox{ for } R\in (-L_2, L_1),\;  t\geq T_0-t_k,
\]
and
\[\mbox{$Z_k(-L_2,t)=Z_k(L_1,t)=0, \; Z_k(R,t)\geq \underline z(R)$ for $R\in [-L_2,L_1]$, $t\geq T_0-t_k$,}\ k\geq k_0.
\]

By a simple comparison argument involving a suitable ODE problem we easily obtain
\[
Z_k(R,t)\leq M:=\max\{ \|v(\cdot, T_0)\|_{L^\infty}, 1\} \mbox{ for $R\in [-L_2,L_1]$, $t\geq T_0-t_k$},\ k\geq 1.
\]
Since $ \frac{N-1}{R+p_k(t_k+t)}+p_k\to p_0$ uniformly as $k\to\infty$, we may apply the parabolic $L^p$ estimate to the equations satisfied by $Z_k$
to  conclude that, for any $p>1$
and $T>0$, there exists $C_1>0$ such that, for all large $k\geq k_0$, say $k\geq k_1$,
\[
\|Z_k\|_{W^{2,1}_p([-L_2, L_1]\times [-T,T])}\leq C_1.
\]
It then follows from the Sobolev embedding theorem that, for every $\alpha\in (0,1)$ and all  $k\geq k_1$,
\[
\|Z_k\|_{C^{1+\alpha, (1+\alpha)/2}([-L_2, L_1]\times [-T,T])}\leq C_2
\]
for some constant $C_2$ depending on $C_1$ and $\alpha$. Let $\tilde\alpha\in (0, \alpha)$. Then by compact embedding and a well known diagonal process, we can find a subsequence of $\{Z_k\}$, still denoted by itself for the seek of convenience, such that
\[\mbox{
$Z_k(R,t)\to Z(R,t)$ as $k\to\infty$ in $C_{loc}^{1+\tilde\alpha, (1+\tilde\alpha)/2}([-L_2,L_1]\times \mathbb R)$.}
\]
From the equations satisfied by $Z_k$ we obtain
\[
Z_t=Z_{RR}+p_0Z_R+Z(1-Z) \mbox{ for } R\in (-L_2, L_1),\; t\in \mathbb R,
\]
and
\[
Z(-L_2, t)=Z(L_1,t)=0,\; M\geq Z(R, t)\geq \underline z(R) \mbox{ for } R\in [-L_2, L_1],\; t\in\mathbb R.
\]

We show that $Z(R,t)\equiv z(R)$. Indeed, if we denote by $\underline Z$ the unique solution of
\[
z_t=z_{RR}+p_0z_R+z(1-z) \mbox{ for } R\in (-L_2, L_1), \; t>0
\]
with boundary conditions $z(-L_2,t)=z(L_1,t)=0$ and initial condition $z(R,0)=\underline z(R)$, while let $\overline Z$ be the unique solution to this problem but with initial condition replaced by $z(R,0)=M$, then clearly
\[
\lim_{t\to\infty} \underline Z(R,t)=\lim_{t\to\infty}\overline Z(R,t)=z(R).
\]
On the other hand, for any $s>0$, by the comparison principle we have
\[
\underline Z(R, t+s)\leq Z(R,t)\leq \overline Z(R, t+s) \mbox{ for }  R\in [-L_2, L_1],\; t\geq -s.
\]
Letting $s\to\infty$ we obtain $z(R)\leq Z(R,t)\leq z(R)$.  We have thus proved $Z(R,t)\equiv z(R)$ and hence
\[
z_k(0, t_k)=Z_k(0,0)\to Z(0,0)=z(0) \mbox{ as } k\to\infty.
\]
This proves \eqref{z_k} and the proof of Lemma 2.2 is complete.
\end{proof}

We now start to construct some auxiliary functions by modifying the unique solution $(U,V)$ of
\eqref{semi-wave system} with $c=c_{\mu_1}^*$.
 Firstly, for any given small $\varepsilon\in(0,1)$ we consider the following perturbed problem
\bea\label{perturbed system}
\begin{cases}
{cU'+d U''+rU(1+\varepsilon-U-kV)=0} \mbox{ for } -\infty<\xi<0,\\
{cV'+V''+V(1-\varepsilon-V-hU)=0} \mbox{ for} -\infty<\xi<\infty,\\
U(-\infty)=1+\varepsilon,\; U(0)=0,\; U_{\varepsilon}'(0)=-\mu_1/c,\; U'(\xi)<0=U(-\xi) \mbox{ for } \xi<0\\
V(-\infty)=0,\quad V(+\infty)=1-\varepsilon,\quad V'(\xi)>0 \mbox{ for } \xi\in\R.
\end{cases}
\eea
{Taking $U=(1+\varepsilon)\widehat{U}$ and $V=(1-\varepsilon)\widehat{V}$, then
$(\widehat{U}, \widehat{V})$ satisfies \eqref{semi-wave system} for $k$ and $h$ replaced by some $\hat{k}_\varepsilon$ and $\hat{h}_\varepsilon$
with $0<\hat{k}_\varepsilon<1<\hat{h}_\varepsilon$, and $\hat{k}_\varepsilon\to k$, $\hat{h}_\varepsilon\to h$
as $\varepsilon\to 0$. Hence by}
Theorem~A, there exists a unique $c=c^{\varepsilon}_{\mu_1}>0$ such that
\eqref{perturbed system} with $c=c^{\varepsilon}_{\mu_1}$ admits a unique solution
$(U_{\varepsilon},V_{\varepsilon})$. As in \cite{DWZ}, $(U_\varepsilon, V_\varepsilon)$ and $c^\varepsilon_{\mu_1}$ depends continuously on $\varepsilon$, and in particular, $c^{\varepsilon}_{\mu_1}\to c^{*}_{\mu_1}$ as $\varepsilon\to 0$. Moreover,
as in the proof of Lemma 2.5 in \cite{DWZ}, we have the asymptotic expansion
\bea\label{AS-V}
V_{\varepsilon}(\xi)=Ce^{\mu\xi}(1+o(1)), \quad V_{\varepsilon}'(\xi)=C\mu e^{\mu\xi}(1+o(1))\ \mbox{as $\xi\to-\infty$}
\eea
for some $C>0$, where
\beaa
\mu=\mu(\varepsilon):=\frac{-c^{\varepsilon}_{\mu_1}+\sqrt{(c^{\varepsilon}_{\mu_1})^2+4[h(1+\varepsilon)-1+\varepsilon]}}{2}>0.
\eeaa

Next we modify $(U_{\varepsilon}(\xi),V_{\varepsilon}(\xi))$ to obtain the required auxiliary functions. The modification of $V_{\varepsilon}$ is rather involved, and for simplicity,  we do that for  $\xi\geq0$ and $\xi\leq0$ separately.

We first consider the case $\xi\geq 0$.
For fixed $\varepsilon\in(0,1)$ sufficiently small, we define
\bea\label{Q plus}
Q^{\varepsilon}_+(\xi)=\begin{cases}
V_{\varepsilon}(\xi) &\mbox{ for } 0\leq\xi\leq\xi_0,\\
V_{\varepsilon}(\xi)-\delta(\xi-\xi_0)^2V_{\varepsilon}(\xi_0)& \mbox{ for } \xi_0\leq\xi\leq\xi_0+1,
\end{cases}
\eea
where $\xi_0=\xi_0(\varepsilon)>0$ is determined later and
\bea\label{def-delta}
\delta=\delta(\varepsilon):=\frac{\varepsilon}{4+2c^{\varepsilon}_{\mu_1}}\in(0,1).
\eea
It is straightforward to see that $Q^{\varepsilon}_+\in C^1([0,\xi_0+1])$.
The following result will be useful later.
\begin{lemma}\label{lem:Q+}
For any small $\varepsilon>0$, there exist $\xi_0=\xi_0(\varepsilon)>0$ and
$\xi_1=\xi_1(\varepsilon)\in(\xi_0,\xi_0+1)$ such that $\lim_{\varepsilon\to 0}\xi_0(\varepsilon)=\infty$ and
\[\mbox{
$(Q_+^{\varepsilon})'(\xi_1)=0$,
$(Q_+^{\varepsilon})'(\xi)>0$ for $\xi\in[0,\xi_1),$}
\]
\bea\label{Q+ineq}
c^{\varepsilon}_{\mu_1}(Q^{\varepsilon}_{+})'+(Q^{\varepsilon}_{+})''+Q^{\varepsilon}_+(1-Q^{\varepsilon}_{+})\geq0\quad
\mbox{for $\xi\in[0,\xi_1]\backslash\{\xi_0\}$.}
\eea
Moreover, there exists $s^\varepsilon\in(0,s^*_{\mu_2})$ such that
$s^\varepsilon\to s^*_{\mu_2}$ as $\varepsilon\to0^+$ and
\bea\label{connect+}
Q^{\varepsilon}_+(\xi_1)=q_{s^\varepsilon}(-z(s^\varepsilon)),
\eea
where $z(s^\varepsilon)$ and $q_{s^\varepsilon}$ are defined in Lemma~\ref{lem:q} with $s=s^\varepsilon$.
\end{lemma}
\begin{proof} For convenience of notation we will write $Q^{\varepsilon}_+=Q$.
Since $V_{\varepsilon}(\infty)= 1-\varepsilon$, $V_{\varepsilon}'(\infty)=0$ and $V_\varepsilon'>0$,
we can choose $\xi_0=\xi_0(\varepsilon)\gg1$ such that $\lim_{\varepsilon\to 0}\xi_0(\varepsilon)=\infty$ and
\beaa
&&1-2\varepsilon\leq V_{\varepsilon}(\xi)\leq 1-\varepsilon,\quad V_{\varepsilon}'(\xi+1)-2\delta V_{\varepsilon}(\xi)<0
\quad\mbox{for all $\xi\geq \xi_0$}.
\eeaa
In particular, we have
\beaa
Q'(\xi_0+1)=V_{\varepsilon}'(\xi_0+1)-2\delta V_{\varepsilon}(\xi_0)<0.
\eeaa
Note that $Q'(\xi_0)=V_{\varepsilon}'(\xi_0)>0$. By the continuity of $Q'$,
we can find $\xi_1=\xi_1(\xi_0)\in(\xi_0,\xi_0+1)$ such that
\bea\label{xi 1}
Q'(\xi_1)=0,\quad Q'(\xi)>0\ \mbox{for all $\xi\in[\xi_0,\xi_1)$}.
\eea
Hence we have $Q'>0$ in $[0,\xi_1)$ since
$Q'=V_{\varepsilon}'>0$ in $[0,\xi_0]$.

We now prove  \eqref{Q+ineq}.
For $\xi\in[0,\xi_0)$, we have $Q=V_{\varepsilon}$. Using $U_{\varepsilon}(\xi)\equiv0$ for $\xi\geq0$
and the second equation of \eqref{perturbed system}, it is straightforward to see that the inequality in \eqref{Q+ineq} holds
 for $\xi\in[0,\xi_0)$.
For $\xi\in(\xi_0,\xi_1]$, direct computation gives us
\beaa
Q'(\xi)=V'_{\varepsilon}(\xi)-2 \delta (\xi-\xi_0)V_{\varepsilon}(\xi_0),\quad Q''(\xi)=V''_{\varepsilon}(\xi)-2 \delta V_{\varepsilon}(\xi_0).
\eeaa
Hence
\beaa
&&c^{\varepsilon}_{\mu_1}Q'+Q''+Q(1-Q)\\
&=& c^{\varepsilon}_{\mu_1}V'_{\varepsilon}(\xi)-2 c^{\varepsilon}_{\mu_1} \delta (\xi-\xi_0)V_{\varepsilon}(\xi_0)
+V''_{\varepsilon}(\xi)-2\delta V_{\varepsilon}(\xi_0)\\
&&+\big[V_{\varepsilon}(\xi)-\delta(\xi-\xi_0)^2V_{\varepsilon}(\xi_0)\big]\big[1-V_{\varepsilon}(\xi)+\delta(\xi-\xi_0)^2V_{\varepsilon}(\xi_0)\big]
\eeaa
Using $0\leq \xi-\xi_0\leq \xi_1-\xi_0<1$, $1>V_{\varepsilon}(\xi)>V_{\varepsilon}(\xi_0)$  for $\xi\in[\xi_0,\xi_1]$, the identity
\beaa
c^{\varepsilon}_{\mu_1}V'_{\varepsilon}+V''_{\varepsilon}=-(1-\varepsilon-V_{\varepsilon})V_{\varepsilon},
\eeaa
and \eqref{def-delta},
we deduce
\beaa
&&c^{\varepsilon}_{\mu_1}Q'+Q''+Q(1-Q)\\
&\geq& \varepsilon V_{\varepsilon}(\xi) -2c^{\varepsilon}_{\mu_1} \delta V_{\varepsilon}(\xi_0)-2\delta V_{\varepsilon}(\xi_0)
-\delta V_{\varepsilon}(\xi_0)(1-V_{\varepsilon}(\xi))-\delta^2 V_{\varepsilon}^2(\xi_0)\\
&\geq& V_{\varepsilon}(\xi_0)(\varepsilon-2c^{\varepsilon}_{\mu_1}\delta-4\delta)\\
&=& 0\qquad \mbox{ for } \xi\in[\xi_0,\xi_1].
\eeaa

To complete the proof, it remains to show the existence of $s^{\varepsilon}$.
Note that
$$Q(\xi_0)=V_{\varepsilon}(\xi_0)\in[1-2\varepsilon,1-\varepsilon].$$
By \eqref{xi 1}, we have
\bea\label{Q(xi1)}
1-2\varepsilon\leq Q(\xi_0) \leq Q(\xi_1)\leq V_{\varepsilon}(\xi_1) \leq 1-\varepsilon.
\eea

By Lemma~\ref{lem:q}, $q_s(-z(s))$ is a continuous and increasing function of $s$ for $s\in (0, s^*_{\mu_2})$, and $q_s(-z(s))\to 1$ as $s\to s^*_{\mu_2}$. Therefore, in view of \eqref{Q(xi1)}, for each small $\varepsilon>0$ there exists $s^\epsilon\in (0, s^*_{\mu_2})$ such that
{\[
Q(\xi_1)=q_{s^\varepsilon}(-z(s^\varepsilon)).
\]}
Moreover, $s^{\varepsilon}\to s^*_{\mu_2}$ as $\varepsilon\to 0$.
Thus \eqref{connect+} holds.
The proof of Lemma~\ref{lem:Q-} is now complete.
\end{proof}

We now consider the case $\xi\leq 0$. We define
\bea\label{Q minus}
Q^{\varepsilon}_{-}(\xi):=\begin{cases}
V_{\varepsilon}(\xi) & \mbox{ for } \xi_2\leq\xi\leq0,\\
V_{\varepsilon}(\xi)+\gamma(\xi-\xi_2)V_{\varepsilon}(\xi_2)& \mbox{ for } -\infty<\xi\leq \xi_{2},
\end{cases}
\eea
where
\beaa
\gamma(\xi)=\gamma(\xi;\lambda):=-\big(e^{\lambda\xi}+e^{-\lambda\xi}-2\big),
\eeaa
with $\lambda>0$ and $\xi_2<0$ to be determined below.

\begin{lemma}\label{lem:Q-}
Let $\varepsilon>0$ be sufficiently small and $(U_\varepsilon,V_\varepsilon)$ be the solution of \eqref{perturbed system}
with $c=c^{\varepsilon}_{\mu_1}$.
Then there exist $\lambda=\lambda(\varepsilon)>0$ sufficiently small and $\xi_2=\xi_2(\varepsilon)<0$ such that
$V_{\varepsilon}(\xi_2)=Q^{\varepsilon}_{-}(\xi_2)<\varepsilon$ and
\bea\label{property-Q-}
Q^{\varepsilon}_{-}\in C^1((-\infty,0])\cap C^2((-\infty,0]\backslash\{\xi_2\}),\ (Q^{\varepsilon}_{-})'(\xi)>0\
\mbox{for all $\xi<0$}.
\eea
Moreover, there exists a unique $\xi_3\in(-\infty,\xi_2)$ depending on $\xi_2$ and $\lambda$ such that $Q^{\varepsilon}_{-}(\xi_3)=0$
and the following inequality holds:
\bea\label{Q-ineq}
c^{\varepsilon}_{\mu_1} (Q^{\varepsilon}_{-})'+(Q^{\varepsilon}_{-})''+Q^{\varepsilon}_{-} (1-Q^{\varepsilon}_{-}-hU_{\varepsilon})\geq0\
\mbox{for $\xi\in(\xi_3,0)\backslash\{\xi_2\}$}.
\eea
\end{lemma}
\begin{proof}
We write $Q^{\varepsilon}_-=Q$ for convenience of notation.
Using $\gamma'(0)=0$, it is straightforward to see that
\beaa
Q\in C^1((-\infty,0])\cap C^2((-\infty,0]\backslash\{\xi_2\})
\eeaa
for any choice of $\xi_2<0$.
Since $V'_{\varepsilon}>0$ in $\mathbb{R}$ and $\gamma'(\xi)>0$ for $\xi<0$, we have
\beaa
Q'(\xi)=\begin{cases}
V'_{\varepsilon}(\xi)>0 &\mbox{ for } \xi_2\leq\xi\leq0,\\
V'_{\varepsilon}(\xi)+\gamma'(\xi-\xi_2)V_{\varepsilon}(\xi_2)>0 &\mbox{ for } \xi\leq\xi_2.
\end{cases}
\eeaa
Hence \eqref{property-Q-} holds for any choice of $\xi_2<0$.

For any given $\lambda>0$, we take $K_{\lambda}>0$ such that
\bea\label{for xi3}
e^{ K_{\lambda}\lambda}+e^{- K_{\lambda}\lambda}-2>e^{-K_{\lambda}\mu},
\eea
where $\mu>0$ is  given in \eqref{AS-V}.
By \eqref{AS-V}, we have
\beaa
\frac{V_{\varepsilon}(\xi_2-K_{\lambda})}{V_{\varepsilon}(\xi_2)}\to e^{-K_{\lambda}\mu}\quad \mbox{as $\xi_2\to-\infty$}.
\eeaa
Together with \eqref{for xi3}, and $(U_{\varepsilon},V_{\varepsilon})(-\infty)=(1+\varepsilon,0)$,
we can take $\xi_2=\xi_2(\lambda)$ close to $-\infty$ such that
\bea
&&Q(\xi_2-K_{\lambda})=
                  V_{\varepsilon}(\xi_2)\left[\frac{V_{\varepsilon}(\xi_2-K_{\lambda})}{V_{\varepsilon}(\xi_2)}
                  -(e^{ K_{\lambda}\lambda}+e^{- K_{\lambda}\lambda}-2)\right]<0,\notag\\
&&V_{\varepsilon}(\xi_2)<\min\left\{\varepsilon,\frac{\varepsilon}{-\gamma(-K_{\lambda})}\right\},\quad U_{\varepsilon}(\xi_2)>1. \label{ineq V-xi2}
\eea
On the other hand, since $Q(\xi_2)=V_{\varepsilon}(\xi_2)>0$, we can apply the intermediate value theorem
to obtain $\xi_3\in(\xi_2-K_{\lambda},\xi_2)$ such that $Q(\xi_3)=0$. Such $\xi_3$ is unique because of the monotonicity of $Q$.


Next we show that, if $\lambda>0$ has been chosen small enough,
with the above determined $\xi_2$ and $\xi_3$,  \eqref{Q-ineq} holds.
To do this, we consider \eqref{Q-ineq} for $\xi\in(\xi_3,\xi_2)$ and $\xi\in[\xi_2,0)$ separately.

For $\xi\in(\xi_3,\xi_2)$, we write $V_{\varepsilon}=V_{\varepsilon}(\xi)$, $\gamma=\gamma(\xi-\xi_2)$ and obtain
\beaa
&&c^{\varepsilon}_{\mu_1}Q'+Q''+Q(1-Q-hU_{\varepsilon})\\
&=&c^{\varepsilon}_{\mu_1}V'_{\varepsilon}+V''_{\varepsilon}+V_{\varepsilon}(\xi_2)\big[c^{\varepsilon}_{\mu_1}\gamma'+\gamma''\big]
    +\big[V_{\varepsilon}+\gamma V_{\varepsilon}(\xi_2)\big]\big[1-V_{\varepsilon}-\gamma V_{\varepsilon}(\xi_2)-hU_{\varepsilon}\big]\\
&=&-V_{\varepsilon}(1-\varepsilon-V_{\varepsilon}-hU_{\varepsilon})+V_{\varepsilon}(\xi_2)\big[c^{\varepsilon}_{\mu_1}\gamma'+\gamma''\big]
    +\big[V_{\varepsilon}+\gamma V_{\varepsilon}(\xi_2)\big]\big[1-V_{\varepsilon}-\gamma V_{\varepsilon}(\xi_2)-hU_{\varepsilon}\big]\\
&\geq& \varepsilon V_{\varepsilon}+V_{\varepsilon}(\xi_2)\big[c^{\varepsilon}_{\mu_1}\gamma'+\gamma''\big]
     -\gamma V_{\varepsilon}(\xi_2) \big[hU_{\varepsilon}+\gamma V_{\varepsilon}(\xi_2)-1\big].
\eeaa
By \eqref{ineq V-xi2}, for $\xi\in(\xi_3,\xi_2)$,
\beaa
U_{\varepsilon}\geq 1,\;
0>\gamma V_{\varepsilon}(\xi_2)=\gamma(\xi-\xi_2)V_{\varepsilon}(\xi_2)\geq\gamma(-K_{\lambda})V_{\varepsilon}(\xi_2)>-\varepsilon.
\eeaa
It follows that
\begin{equation}\label{key ineq}
\begin{array}{ll}
&c^{\varepsilon}_{\mu_1}Q'+Q''+Q(1-Q-hU_{\varepsilon})\smallskip\\
\geq& \varepsilon V_{\varepsilon}+V_{\varepsilon}(\xi_2)\big[c^{\varepsilon}_{\mu_1}\gamma'+\gamma''\big]-\gamma V_{\varepsilon}(\xi_2)\big[h-\varepsilon-1\big] \mbox{ for } \xi_3<\xi<\xi_2.
\end{array}
\end{equation}
Using \eqref{AS-V}, we see that the right side of \eqref{key ineq} is nonnegative if the following inequality holds:
\bea\label{key ineq2}
\varepsilon e^{\mu(\xi-\xi_2)}+c^{\varepsilon}_{\mu_1}\gamma'+\gamma''-[h-\varepsilon-1]\gamma>0 \mbox{ for } \xi_3<\xi<\xi_2.
\eea
We shall show that \eqref{key ineq2} indeed holds provided that $\lambda>0$ has been chosen small enough.
To check this, for $t=\xi_2-\xi\geq 0$ we define
\be\label{F}
\begin{array}{ll}
F(t):=&\varepsilon e^{-\mu t}
     -c^{\varepsilon}_{\mu_1}\lambda(e^{-\lambda t}-e^{\lambda t})-\lambda^2(e^{-\lambda t}+e^{\lambda t})\smallskip
      \\&\quad+[h-\varepsilon-1](e^{-\lambda t}+e^{\lambda t}-2).
\end{array}
\ee
By Lemma~\ref{lem:F} below, we can take small $\lambda$ depending only on $\varepsilon$ such that $F(t)>0$ for all $t\geq0$.
This implies \eqref{key ineq2}, and so \eqref{Q-ineq}  holds for $\xi\in(\xi_3,\xi_2]$.

For $\xi\in(\xi_2,0)$, we have $Q'(\xi)=V_{\varepsilon}(\xi)$. From \eqref{perturbed system}, it is straightforward to see that \eqref{Q-ineq}
holds for $\xi\in (\xi_2, 0)$.
 This completes the proof.
\end{proof}

\begin{lemma}\label{lem:F}
Let $\varepsilon>0$ and $F:[0,\infty)\to\mathbb{R}$ be defined by \eqref{F}.
Then $F(t)>0$ for all $t\geq0$ as long as $\lambda>0$ is small enough.
\end{lemma}
\begin{proof} The argument is similar to \cite[Lemma 3.3]{DWZ}.
 Let $\kappa:=h-\varepsilon-1$. Note that $\kappa>0$ since $h>1$.
By direct computations,
\beaa
&&F'(t)=-\varepsilon\mu e^{-\mu t}+\lambda^2c^{\varepsilon}_{\mu_1}(e^{\lambda t}
+e^{-\lambda t})+(\kappa\lambda-\lambda^3)(e^{\lambda t}-e^{-\lambda t}),\\
&&F''(t)=\varepsilon\mu^2 e^{-\mu t}+\lambda^3c^{\varepsilon}_{\mu_1}(e^{\lambda t}
-e^{-\lambda t})+(\kappa\lambda^2-\lambda^4)(e^{\lambda t}+e^{-\lambda t}),
\eeaa
where $\mu>0$ is given in \eqref{AS-V}.
By taking
\beaa
\lambda\in\left(0,\min\left\{\sqrt{\frac{\varepsilon}{2}},\ \sqrt{\frac{\varepsilon\mu}{2c^{\varepsilon}_{\mu_1}}},\ \sqrt{\kappa}\right\}\right),
\eeaa
we have $F(0)>0$, $F'(0)<0$, $F'(\infty)=\infty$ and $F''(t)>0$ for $t\geq0$.
If follows that  $F$ has a unique minimum point $t=t_{\lambda}$.
Consequently, to finish the proof of Lemma~\ref{lem:F}, it suffices to show the following:
\bea\label{goal-lem F}
\mbox{$F(t_{\lambda})\geq0$ as long as $\lambda>0$ is small}.
\eea
By direct calculation,
$F'(t_{\lambda})=0$ implies that
\bea\label{eq-t-lambda}
\varepsilon\mu e^{-\mu t_{\lambda}}=\lambda^2c^{\varepsilon}_{\mu_1}(e^{\lambda t_{\lambda}}
+e^{-\lambda t_{\lambda}})+(\kappa\lambda-\lambda^3)(e^{\lambda t_{\lambda}}-e^{-\lambda t_{\lambda}}).
\eea
From \eqref{eq-t-lambda}, we easily deduce $t_\lambda\to \infty$ as $\lambda\to 0$, for otherwise, the left hand side of \eqref{eq-t-lambda}
is {bounded below} by a positive constant while the right hand side converges to 0 as $\lambda\to 0$ along some sequence.
Multiplying $t_\lambda$ to both sides of \eqref{eq-t-lambda} and we obtain, by a similar consideration, that
$\lambda t_\lambda$ is bounded from above by a positive constant as $\lambda\to 0$. It then follows
that
\[
\kappa \lambda t_\lambda e^{\lambda t_\lambda}\to 0 \mbox{ as } \lambda\to 0,
\]
which implies $\lambda t_\lambda\to 0$ as $\lambda\to 0$. We thus obtain
\bea\label{lambda-behavior}
t_{\lambda}\to\infty,\quad \lambda t_{\lambda}\to0\quad\mbox{as $\lambda\to0^+$}.
\eea
It follows that
\bea\label{lim-ratio}
\lim_{\lambda\to 0^+}\frac{e^{\lambda t_{\lambda}}-e^{-\lambda t_{\lambda}}}{ 2\lambda t_{\lambda} }
=\lim_{\lambda\to 0^+}\frac{e^{\lambda t_{\lambda}}+e^{-\lambda t_{\lambda}}}{ 2 }
=1.
\eea

We now prove \eqref{goal-lem F}. Substituting \eqref{eq-t-lambda} into $F$, we have
\beaa
F(t_{\lambda})&=&(e^{\lambda t_{\lambda}}-e^{-\lambda t_{\lambda}})
\left(c^{\varepsilon}_{\mu_1} \lambda-\frac{\lambda^3}{\mu}+\frac{\kappa}{\mu}\lambda\right)
\\
&&+\lambda^2(e^{\lambda t_{\lambda}}+e^{-\lambda t_{\lambda}})\left(\frac{c^{\varepsilon}_{\mu_1}}{\mu}-1\right)+
\kappa(e^{\lambda t_{\lambda}}+e^{-\lambda t_{\lambda}}-2)
\eeaa
Using \eqref{lambda-behavior} and \eqref{lim-ratio}, for small $\lambda>0$,
\beaa
F(t_{\lambda})&\geq& 2\lambda t_{\lambda}[1+o(1)]\left(c^{\varepsilon}_{\mu_1}\lambda-\frac{\lambda^3}{\mu}+\frac{\kappa}{\mu}\lambda\right)+
\lambda^2[2+o(1)]\left(\frac{c^{\varepsilon}_{\mu_1}}{\mu}-1\right)\\
&=&2 \lambda^2t_\lambda\left[ c_{\mu_1}^{\varepsilon}+\frac{\kappa}{\mu}+o(1)\right]\\
&>& 0.
\eeaa
 This completes the proof.
\end{proof}

Combining \eqref{Q plus} and \eqref{Q minus} we now define
\bea\label{Q}
\widehat{Q}^{\varepsilon}(\xi):=\begin{cases}
Q^\varepsilon_{+}(\xi) &\mbox{ for } \xi\in[0,\xi_1],\\
Q^\varepsilon_{-}(\xi) &\mbox{ for } \xi\in[\xi_3,0],\\
0 &\mbox{ for } \xi\in(-\infty,\xi_3],
\end{cases}
\eea
where  $\xi_1>0$ and $\xi_3<0$ are given in Lemmas~\ref{lem:Q+} and \ref{lem:Q-}, respectively.
We then define
\beaa\label{W}
W_{\varepsilon}(r):=\begin{cases}
0 &\mbox{ for } r\in[\xi_1+z(s^{\varepsilon}),\infty).\\
q_{s^\varepsilon}(r-\xi_1-z(s^{\varepsilon})) &\mbox{ for } r\in[\xi_1, \xi_1+z(s^{\varepsilon})],\\
\widehat{Q}^{\varepsilon}(r) & \mbox{ for } r\in(-\infty,\xi_1],
\end{cases}
\eeaa
with $s^{\varepsilon}\in (0, s^*_{\mu_2})$ given in Lemma~\ref{lem:Q+}.
Then clearly $W_{\varepsilon}\in C(\mathbb{R})$ has  compact support $[\xi_3,\xi_1+z(s^{\varepsilon})]$.

\bigskip

We are now ready to describe the  conditions in Theorem 1 on the initial functions $u_0$ and $v_0$.
Since $s^{\varepsilon}\to s^*_{\mu_2}>c^*_{\mu_1}$ and $\xi_0(\varepsilon)\to \infty$ as $\varepsilon\to 0$,
where $\xi_0(\varepsilon)$ is defined in Lemma~\ref{lem:Q+},
 we can fix $\varepsilon_0>0$ small so that
\be\label{ep0}
s^{\varepsilon}-c^{\varepsilon}_{\mu_1}>\frac{N-1}{\xi_0(\varepsilon)}>0 \mbox{ for all } \varepsilon\in (0,\varepsilon_0],
\ee
where $N$ is the space dimension.
Our first condition  is
\smallskip

{\bf (B1):}  For some $z_0>0$ and small $\varepsilon_0>0$ as above,
\[
u_0(r)\leq U_{\varepsilon_0}(r-z_0), \ v_0(r)\geq W_{\varepsilon_0}(r-z_0) \mbox{ for } r\geq 0.
\]
We note that {\bf (B1)} implies
\[
s_1^0\leq z_0 \mbox{ and } s_2^0\geq \xi_1(\varepsilon_0)+z(s^{\varepsilon_0})+z_0.
\]

\smallskip

Our second condition is
\smallskip

{\bf (B2):}  $s_1^{0}\geq R^*\sqrt{\frac{d}{r(1-k)}}$, where $R^*>0$ is defined in Corollary 1.
\smallskip
\\
Since $\limsup_{t\to\infty} v(r,t)\leq 1$ uniformly in $r\in [0, s_2(t)]$, it is easy to see that  {\bf (B2)} guarantees
$s_{1,\infty}=\lim_{t\to\infty} s_1(t)=\infty$ {(see also the proof of Theorem~\ref{prop-trichotomy})}.
\bigskip

We are now ready to prove Theorem 1, which we restate as
\begin{thm}\label{thm: speed}
Suppose that \eqref{h-k}, \eqref{A2}, {\bf (B1)} and {\bf (B2)} hold. Then the solution
$(u,v,s_1,s_2)$  of {\bf(P)} satisfies
\beaa
\lim_{t\to\infty}\frac{s_1(t)}{t}=c^*_{\mu_1},\quad \lim_{t\to\infty}\frac{s_2(t)}{t}=s^*_{\mu_2},
\eeaa
and for every small $\epsilon>0$, \eqref{u-v-1}, \eqref{v-2} hold.

\end{thm}

Before giving the proof of Theorem 3, let us first observe how Corollary 1 follows easily from Theorem 3. It suffices to show that assumptions (i) and (ii) in Corollary 1 imply {\bf (B1)} and {\bf (B2)}.
Recall that
\[
U_{\varepsilon_0}(-\infty)=1+\varepsilon_0 \mbox{ and } U_{\varepsilon_0}(\xi)=0 \mbox{ for } \xi\geq 0.
\]
Therefore, for fixed $s_1^0\geq  R^*\sqrt{\frac{d}{r(1-k)}}$, there exists $C_1>0$ large such that
\[
U_{\varepsilon_0}(r-z_0)\geq 1 \mbox{ for } r\in [0, s_1^0],\; z_0\geq C_1.
\]
Hence for any given $u_0$ satisfying (i) of Corollary 1, {\bf (B2)} and the first inequality in {\bf (B1)} are satisfied
if we take $z_0\geq C_1$.

From the definition of $W_{\varepsilon_0}$ we see that
\[
W_{\varepsilon_0}(\xi)\leq 1 \mbox{ for $\xi\in \mathbb{R}^1$,  and } W_{\varepsilon_0}(\xi)=0 \mbox{ for } \xi\not\in [\xi_3, \xi_1+
z(s^{\varepsilon_0})].
\]
If we take
\[
C_0:=\max\big\{ C_1, \;\xi_1+z(s^{\varepsilon_0})-\xi_3\big\},
\]
and  for $x_0\geq C_0$ and $L\geq C_0$, we let $z_0:=x_0-\xi_3$, then
\[
z_0\geq x_0\geq C_1,\;
[\xi_3, \xi_1+
z(s^{\varepsilon_0})]\subset [x_0-z_0, x_0+L-z_0],
\]
and hence $W_{\varepsilon_0}(r-z_0)=0$ for $r\not\in [x_0, x_0+L]$. Thus when (ii) in Corollary 1 holds, we have
\[
v_0(r)\geq W_{\varepsilon_0}(r-z_0) \mbox{ for } r\geq 0,
\]
which is the second inequality in {\bf (B1)}. This proves what we wanted.

\begin{proof}[Proof of Theorem 3]

We break the rather long proof into 4 steps.

\smallskip

\noindent
{\bf Step 1:} We show
\be\label{s1-s2-1}
\limsup_{t\to\infty}\frac{s_1(t)}{t}\leq  c^{\varepsilon_0}_{\mu_1},\quad
      \liminf_{t\to\infty}\frac{s_2(t)}{t}\geq s^{\varepsilon_0}-\sigma_{\varepsilon_0},
\ee
where
\[
\sigma_\varepsilon:=\frac{N-1}{\xi_0(\varepsilon)} \mbox{ for } \varepsilon\in (0,\varepsilon_0].
\]

By \eqref{ep0}, we have
\bea\label{sigma}
c^{\varepsilon}_{\mu_1}<s^{\varepsilon}-\sigma_{\varepsilon} \mbox{ for } \varepsilon\in (0,\varepsilon_0].
\eea
We prove  \eqref{s1-s2-1} by constructing suitable functions $(\overline U(r,t), \underline V(r,t), l(t), g(t))$ which satisfy certain differential inequalities that enable us to use a comparison argument to relate them to $(u(r,t),v(r,t), s_1(t), s_2(t))$.
Set
\beaa
&&l(t):=c^{\varepsilon_0}_{\mu_1}t+z_0,\quad g(t):=\left(s^{\varepsilon_0}-\sigma_{\varepsilon_0}\right)t+\xi_1+z(s^{{\varepsilon_0}})+z_0,\\
&&\overline{U}(r,t):=U_{{\varepsilon_0}}(r-l(t)) \mbox{ for } r\in[0,l(t)],\ t\geq0,\\
&&\underline{V}(r,t):=\begin{cases}
\widehat{Q}^{\varepsilon_0}(r-l(t)) &\mbox{ for } r\in[0,l(t)+\xi_1],\ t\geq0,\\
\widehat{Q}^{\varepsilon_0}(\xi_1) &\mbox{ for } r\in[l(t)+\xi_1,g(t)-z(s^{{\varepsilon_0}})],\ t\geq0,\\
q_{s^{\varepsilon_0}}(r-g(t)) &\mbox{ for } r\in[g(t)-z(s^{{\varepsilon_0}}), g(t)],\ t\geq0,
\end{cases}
\eeaa
where $\widehat{Q}^{\varepsilon_0}$ is defined in \eqref{Q} and $\xi_1=\xi_1(\varepsilon_0)$ is given in Lemma~\ref{lem:Q+}.
We note that
\bea\label{xi 1 large}
\xi_1(\varepsilon_0)>\xi_0(\varepsilon_0)=\frac{N-1}{\sigma_{\varepsilon_0}}.
\eea

By the assumption {\bf (B1)}, we have
\bea\label{initial cp}
u(r,0)\leq\overline{U}(r,0) \mbox{ for } r\in[0,s^0_{1}];\
v(r,0)\geq \underline{V}(r,0) \mbox{ for } r\in[0,s^0_{2}].
\eea
We now show the wanted differential inequality for $\overline U$:
\bea\label{bar U}
\overline{U}_t-d\overline{U}_{rr}-\frac{N-1}{r}\overline{U}_{r}-r\overline{U}(1-\overline{U}-k\underline{V})\geq0 \mbox{ for } 0\leq r\leq l(t),\ t>0.
\eea
Using $\overline{U}_{r}=U_{\varepsilon_0}'<0$,
direct computation gives us
\bea\label{bar U ineq}
\begin{array}{rl}
J(r,t):=&\overline{U}_t-d\overline{U}_{rr}-\frac{N-1}{r}\overline{U}_{r}-r\overline{U}(1-\overline{U}-k\underline{V})\smallskip\\
\geq& -c^{\varepsilon_0}_{\mu_1} U_{{\varepsilon_0}}'-d U_{{\varepsilon_0}}''-rU_{{\varepsilon_0}}(1-U_{{\varepsilon_0}}-k\underline{V})\smallskip\\
=& rU_{{\varepsilon_0}}(1+{\varepsilon_0}-U_{{\varepsilon_0}}-kV_{{\varepsilon_0}})-rU_{{\varepsilon_0}}(1-U_{{\varepsilon_0}}-k\underline{V})\smallskip\\
=& rU_{{\varepsilon_0}}({\varepsilon_0}-kV_{{\varepsilon_0}}+k\underline{V}).
\end{array}
\eea

When $\overline{U}>0$, we have $r\leq l(t)$ and so we can divide into two cases:
when $r-l(t)\in(\xi_2,0)$, we have
$$\underline{V}(r,t)=\widehat{Q}^{{\varepsilon_0}}_-(r-l(t))=V_{{\varepsilon_0}}(r-l(t)).$$
Hence from \eqref{bar U ineq} we see that $J(r,t)>0$. When $r-l(t)\in(-l(t),\xi_2)$, by Lemma~\ref{lem:Q-} with $\varepsilon=\varepsilon_0$,
we have
$$ kV_{{\varepsilon_0}}(r-l(t))<V_{{\varepsilon_0}}(r-l(t))<V_{{\varepsilon_0}}(\xi_2)<{\varepsilon_0}.$$
Again we obtain from \eqref{bar U ineq} that $J(r,t)>0$. Hence \eqref{bar U} holds.

We next show 
the wanted differential inequality for $\underline V$:
\bea\label{bar V ineq}
\underline{V}_t-\underline{V}_{rr}-\frac{N-1}{r}\underline{V}_{r}-\underline{V}(1-\underline{V}-h\overline{U})\leq0 \mbox{ for }
 0\leq r\leq g(t),\ t>0.
\eea
We divide the proof into three parts.

(i) For $r\in[0,l(t)+\xi_1]$, using $\underline V_r(r,t)=(\widehat{Q}^{\varepsilon_0})'(r-l(t))\geq0$, Lemma~\ref{lem:Q+} and Lemma~\ref{lem:Q-}
with $\varepsilon=\varepsilon_0$,
\beaa
\underline{V}_t-\underline{V}_{rr}-\frac{N-1}{r}\underline{V}_{r}-\underline{V}(1-\underline{V}-h\overline{U})
\leq-\frac{N-1}{r}\underline V_r\leq0.
\eeaa

(ii) For  $r\in[l(t)+\xi_1, g(t)-z(s^{{\varepsilon_0}})]$, we have $\overline U\equiv 0$ and $\underline{V}\equiv \widehat{Q}^{\varepsilon_0}(\xi_1)<1$.
So clearly
\eqref{bar V ineq} holds.

(iii)
For $r\in(g(t)-z(s^{{\varepsilon_0}}), g(t))$, we observe that $r\geq g(t)-z(s^{{\varepsilon_0}})\geq \xi_1$.
Also, by \eqref{xi 1 large},
we have
\beaa
\sigma_{\varepsilon_0}-\frac{N-1}{r}\geq \sigma_{\varepsilon_0}-\frac{N-1}{\xi_1}>0.
\eeaa
Together with the fact that $(q_{s^{\varepsilon_0}})'(r-g(t))<0$ for $r\in(g(t)-z(s^{{\varepsilon_0}}), g(t))$ and $t>0$, we have
\beaa
&&\underline{V}_t-\underline{V}_{rr}-\frac{N-1}{r}\underline{V}_{r}-\underline{V}(1-\underline{V}-h\overline{U})\\
&=& -g'(t)q_{s^{\varepsilon_0}}'-q_{s^{\varepsilon_0}}''-\frac{N-1}{r}q_{s^{\varepsilon_0}}'
    -q_{s^{\varepsilon_0}}(1-q_{s^{\varepsilon_0}})\\
&=&\left(\sigma_{\varepsilon_0}-\frac{N-1}{r}\right)q_{s^{\varepsilon_0}}'(r-g(t))\\
&\leq& 0.
\eeaa
We have thus proved \eqref{bar V ineq}.

In order to use the comparison principle to compare $(u,v,s_1, s_2)$ with $(\overline U,\underline V, l, g)$,
we note that on the  boundary $r=0$,
\bea\label{left bdry cp}
\overline{U}_r(0,t)<0,\quad \underline{V}_r(0,t)>0 \mbox{ for } t>0.
\eea
Regarding the free boundary conditions, we have
\bea
&&l'(t)=c_{\mu}^{\varepsilon_0}=-\mu_1U_{\varepsilon_0}'(0)=-\mu_1\underline{U}(l(t),t) \mbox{ for } t>0,\label{fbc-l}\\
&&g'(t)=s^{\varepsilon_0}-\sigma_{\varepsilon_0}<s^*_{\mu_2}=-\mu_2 (q_{s}^{\varepsilon_0})'(0)=-\mu_2\underline{V}(g(t),t)
\mbox{ for } t>0.\label{fbc-g}
\eea
By \eqref{initial cp}, \eqref{bar U}, \eqref{bar V ineq}-\eqref{fbc-g},
we can apply the comparison principle (\cite[Lemma 3.1]{Wu2} with minor modifications) to deduce that
\[
s_1(t)\leq l(t),\; u(r,t)\leq \overline U(r,t) \mbox{ for } r\in[0, s_1(t)),\; t>0;
\]
\[
s_2(t)\geq g(t),\; v(r,t)\geq \underline V(r,t) \mbox{ for } r\in [0, g(t)],\; t>0.
\]
In particular,
\[
\limsup_{t\to\infty}\frac{s_1(t)}{t}\leq \lim_{t\to\infty}\frac{l(t)}{t}= c^{\varepsilon_0}_{\mu_1},\quad
      \liminf_{t\to\infty}\frac{s_2(t)}{t}\geq \lim_{t\to\infty}\frac{g(t)}{t}=s^{\varepsilon_0}-\sigma_{\varepsilon_0}.
\]
We have thus proved \eqref{s1-s2-1}.

\bigskip

\noindent{\bf Step 2:}
We refine the definitions of $(\overline U(r,t), \underline V(r,t), l(t), g(t))$ in Step 1
to obtain the improved estimates
 \bea\label{lower estimate}
\limsup_{t\to\infty}\frac{s_1(t)}{t}\leq c^{*}_{\mu_1},\quad \liminf_{t\to\infty}\frac{s_2(t)}{t}\geq s^{*}_{\mu_2}.
\eea

For any given $\varepsilon\in(0,\varepsilon_0)$, we redefine $(l,g,\overline{U},\underline{V})$ as
\beaa
&&l(t)=c^{\varepsilon}_{\mu_1}(t-T_{\varepsilon})+z_1,\quad
      g(t)=(s^{\varepsilon}-\sigma_{\varepsilon})(t-T_{\varepsilon})+\xi_1+z(s^{{\varepsilon}})+z_1,\\
&&\overline{U}(r,t)=U_{\varepsilon}(r-l(t)) \mbox{ for } r\in[0,l(t)],\ t\geq 0,\\
&&\underline{V}(r,t)=\begin{cases}
\widehat{Q}^{\varepsilon}(r-l(t)) & \mbox{ for } r\in[0,l(t)+\xi_1],\ t\geq 0,\\
\widehat{Q}^{\varepsilon}(\xi_1) &\mbox{ for } r\in[l(t)+\xi_1,g(t)-z(s^{\varepsilon})],\ t\geq 0,\\
q_{s^\varepsilon}(r-g(t)) &\mbox{ for } r\in[g(t)-z(s^{\varepsilon}), g(t)],\ t\geq 0,\\
0  &\mbox{ for}  r\in[g(t),\infty),\ t\geq 0,
\end{cases}
\eeaa
where $\xi_1=\xi_1(\varepsilon)$ is given in Lemma~\ref{lem:Q+}
and
$z_1, T_{\varepsilon}\gg1$ are to be determined later.

We want to show that
there exist $z_1\gg1$ and $T_{\varepsilon}\gg1$ such that
\bea\label{T cp}
u(r,T_{\varepsilon})\leq\overline{U}(r,T_{\varepsilon}) \mbox{ for } r\in[0,s_1(T_{\varepsilon})];\
v(r,T_{\varepsilon})\geq\underline{V}(r,T_{\varepsilon}) \mbox{ for } r\in[0,s_2(T_{\varepsilon})].
\eea

Since $\limsup_{t\to\infty}u(r,t)\leq 1$ uniformly in $r$,
  there exists $T_{1,\varepsilon}$ such that
\bea\label{u-behavior}
u(r,t)\leq 1+\varepsilon/2\quad \mbox{for $r\in[0,s_1(t)]$ and $t\geq T_{1,\varepsilon}$}.
\eea
{By 
\eqref{sigma} }
and Lemma~\ref{lem:estimate1}, we can find $0<\nu\ll1$ and then $T_{2,\varepsilon}\gg1$  such that
\bea
&& c_1:=c^{\varepsilon_0}_{\mu_1}+\nu<s^{\varepsilon_0}-\sigma_{\varepsilon_0}-\nu=:c_2,\label{c1-2}\\
&&v(r,t)\geq 1-\varepsilon\quad
           \mbox{for $r\in [c_1 t,c_2 t]$ and $t\geq T_{2,\varepsilon}$}.\label{top}
\eea

We now prove \eqref{T cp} by making use of \eqref{u-behavior} and \eqref{top}.
By the definition of $\underline{V}(r,t)$, we see that
$\|\underline{V}(\cdot,t)\|_{L^{\infty}}<1-\varepsilon$ for all $t>0$. Also, note that
$\underline{V}(\cdot,T_{\varepsilon})=W_{\varepsilon}(\cdot-z_1)$
has  compact support $[\xi_3+z_1,\xi_1+z(s^\varepsilon)+z_1]$, whose length equals to $\xi_1-\xi_3+z(s^{\varepsilon})$ which
 is independent of the choice of $T_{\varepsilon}$.

Next, we show the following claim: there exist $z_1\gg1$ and $T_{\varepsilon}\gg1$ such that
\beaa
[\xi_3+z_1,\xi_1+z(s^\varepsilon)+z_1]\subset[c_1T_\varepsilon,c_2T_\varepsilon].
\eeaa
Since $U_{\varepsilon}(-\infty)=1+\varepsilon$, we can find $T_{3,\varepsilon}\gg 1$ such that
\[
U_{\varepsilon}(r)>1+\frac{\varepsilon}{2} \mbox{ for } r\leq -T_{3,\varepsilon}.
\]
By \eqref{s1-s2-1}, we can find $T_{4,\varepsilon}\gg1$ so that {
\[
s_1(t)\leq (c_{\mu_1}^{\varepsilon_0}+\frac{\nu}{2})t=(c_1-\frac{\nu}{2})t \mbox{ for } t\geq T_{4,\varepsilon}.
\] }
We now take $z_1:=c_1T_\varepsilon-\xi_3$ with  $T_{\varepsilon}>\max\{T_{1,\varepsilon},T_{2,\varepsilon}, T_{4,\varepsilon}\}$
chosen such that
\[
s_1(T_\varepsilon)-c_1T_\varepsilon+\xi_3<-\frac{\nu}{2}T_\varepsilon+\xi_3(\varepsilon)<-T_{3,\varepsilon},
\]
\[
\xi_1+z(s^\varepsilon)+c_1T_\varepsilon-\xi_3<c_2T_\varepsilon.
\]
It follows that
\[
[\xi_3+z_1, \xi_1+z(s^\varepsilon)+z_1]=[c_1T_\varepsilon, \xi_1+z(s^\varepsilon)+c_1T_\varepsilon-\xi_3]\subset[c_1T_\varepsilon,
c_2T_\varepsilon],
\]
and
\[
\overline{U}(r,T_{\varepsilon})=U_\varepsilon(r-z_1)={U_\varepsilon(r-c_1T_\varepsilon+\xi_3)}\geq 1+\frac{\varepsilon}{2}
\mbox{ for } r\leq s_1(T_\epsilon).
\]
Thus we may use \eqref{u-behavior} and \eqref{top} to obtain
\beaa
u(r,T_{\varepsilon})\leq 1+\frac{\varepsilon}{2}\leq\overline{U}(r,T_{\varepsilon}) \mbox{ for } r\in[0,s_1(T_{\varepsilon})],
\eeaa
\beaa
\underline{V}(\cdot,T_{\varepsilon})<{1-\varepsilon}\leq v(\cdot,T_\varepsilon) \mbox{ for } r\in[0, s_2(T_\varepsilon)].
\eeaa
We have thus proved \eqref{T cp}.

It is also easily seen that,
with $t>0$ replaced by $t>T_{\varepsilon}$ and $\varepsilon_0$  replaced by
$\varepsilon$, the inequalities \eqref{bar U} and \eqref{bar V ineq}--\eqref{fbc-g} still hold. Thus we are able to use the comparison principle  as before to deduce
\[
s_1(t)\leq l(t),\; u(r,t)\leq \overline U(r,t) \mbox{ for } r\in[0, s_1(t)),\; t>T_\varepsilon;
\]
\[
s_2(t)\geq g(t),\; v(r,t)\geq \underline V(r,t) \mbox{ for } r\in [0, g(t)],\; t>T_\varepsilon.
\]
In particular,
\beaa
\limsup_{t\to\infty}\frac{s_1(t)}{t}\leq \lim_{t\to\infty}\frac{l(t)}{t}= c^{\varepsilon}_{\mu_1},\quad \liminf_{t\to\infty}\frac{s_2(t)}{t}\geq \lim_{t\to\infty}\frac{g(t)}{t}=s^{\varepsilon}-\sigma_{\varepsilon}.
\eeaa
Since $\varepsilon\in(0,\varepsilon_0)$ is arbitrary, taking $\varepsilon\to 0$ we obtain
 \eqref{lower estimate}.

\bigskip

\noindent{\bf Step 3:} We prove the following conclusions:
\be
\label{s1-u}
\lim_{t\to\infty}\frac{s_1(t)}{t}=c^*_{\mu_1},\; \lim_{t\to\infty} \Big[\max_{r\in[0, (c^*_{\mu_1}-\epsilon)t]} |u(r,t)-1|\Big]= 0.
\ee
\be\label{s2-v}
\lim_{t\to\infty}\frac{s_2(t)}{t}=s^*_{\mu_2},\;
\lim_{t\to\infty}\Big[\max_{r\in[(c^*_{\mu_1}+\epsilon)t, (s^*_{\mu_2}-\epsilon)t]} |v(r,t)-1|\Big]=0.
\ee

We note that for $r\in [l(t)+\xi_1, g(t)-z(s^\varepsilon)]$ and $t>0$,
\begin{align*}
\underline V(r,t)=\widehat{Q}^\varepsilon(\xi_1)&=V_\varepsilon(\xi_1)-\delta(\xi_1-\xi_0)^2V_\varepsilon(\xi_0)&\\
&\geq V_\varepsilon(\xi_0)-\delta V_\varepsilon(\xi_0)=(1-\delta)V_\varepsilon(\xi_0)&\\
&\geq (1-\frac\varepsilon 4)(1-2\varepsilon).&
\end{align*}
Thus for any given $\epsilon>0$ we can choose $\varepsilon^*>0$ small enough so that for all $\varepsilon\in (0, \varepsilon^*]$,
\[
\underline V(r,t)\geq 1-\epsilon \mbox{ for } r\in [l(t)+\xi_1, g(t)-z(s^\varepsilon)],\; t>0.
\]
In view of
\[
\lim_{\varepsilon\to 0} c^\varepsilon_{\mu_1}=c^*_{\mu_1},\; \lim_{\varepsilon\to 0} s^\varepsilon=s^*_{\mu_2},
\]
and the inequality \eqref{sigma}, by further shrinking $\varepsilon^*$ we may also assume that for all $\varepsilon\in (0, \varepsilon^*]$,
\[
c_{\mu_1}^\varepsilon<c_{\mu_1}^*+\frac\epsilon 2,\; s^\varepsilon-\sigma_\varepsilon>s^*_{\mu_2}-\frac\epsilon 2.
\]
Hence for every $\varepsilon\in (0,\varepsilon^*]$ we can find $\tilde T_\varepsilon\geq T_\varepsilon$ such that
\[
[(c^*_{\mu_1}+\epsilon)t, (s^*_{\mu_2}-\epsilon)t]\subset [l(t)+\xi_1, g(t)-z(s^\varepsilon)]\mbox{ for } t\geq \tilde T_\varepsilon.
\]
It follows that
\[
v(r,t)\geq \underline V(r,t)\geq 1-\epsilon \mbox{ for } r\in [(c^*_{\mu_1}+\epsilon)t, (s^*_{\mu_2}-\epsilon)t],\; t\geq \tilde T_\varepsilon,
\;\varepsilon\in (0,\varepsilon^*],
\]
which implies
\be\label{v-lbd}
\liminf_{t\to\infty}\Big[\min_{r\in[(c^*_{\mu_1}+\epsilon)t, (s^*_{\mu_2}-\epsilon)t]} v(r,t)\Big]\geq 1.
\ee

Next we obtain bounds for $(u,v,s_1,s_2)$ from the other side.

By comparison with an ODE upper solution, 
\[
\limsup_{t\to\infty} v(r,t)\leq 1 \mbox{ uniformly for } r\in [0,\infty),
\]
which, combined with \eqref{v-lbd}, yields
\[
\lim_{t\to\infty}\Big[\max_{r\in[(c^*_{\mu_1}+\epsilon)t, (s^*_{\mu_2}-\epsilon)t]} |v(r,t)-1|\Big]=0.
\]
This proves the second identity in \eqref{s2-v}.

{As seen in the proof of Lemma~\ref{lem:estimate1}, we have
\[
 \limsup_{t\to\infty}\frac{s_2(t)}{t}\leq s^{*}_{\mu_2}.
\] }
Combining this with \eqref{lower estimate}, we obtain the first identity in \eqref{s2-v}, namely
\[
\lim_{t\to\infty}\frac{s_2(t)}{t}=s^*_{\mu_2}.
\]

 We next prove \eqref{s1-u}. Consider the problem {\bf(Q)} with initial data in \eqref{ic-Q}
chosen the following way: $
\hat{u}_0=u_0,\; h_0=s^0_1$ and  $\hat{v}_0\in C^2([0,\infty))\cap L^{\infty}((0,\infty))$  satisfies
\eqref{inf v0} and
\bea\label{intial v ineq}
\hat{v}_0(r)\geq v_0(r)\quad \mbox{for $r\in[0,s^0_2]$}.
\eea
We denote its unique solution by $(\hat u, \hat v, h)$.
Then by {\bf (B2)} and Theorem 4.4 in \cite{DL2}, we have $h_{\infty}=\infty$.
Moreover, it follows from Theorem B that
\[
\lim_{t\to\infty}\frac{h(t)}{t}=c^*_{\mu_1}.
\]
Due to \eqref{intial v ineq}, we can apply the comparison principle (\cite[Lemma 3.1]{Wu2} with minor modifications) to derive
\beaa
s_1(t)\geq h(t),\quad u(r,t)\geq \hat u(r,t) \mbox{  for $r\in [0, h(t)]$, $t\geq0$},
\eeaa
which in particular implies
\bea\label{estimate-1}
\liminf_{t\to\infty}\frac{s_1(t)}{t}\geq c^{*}_{\mu_1}.
\eea
Moreover, by the definition of $\psi_\delta$ and $\underline h(t)$ in \cite{DWZ}, and the estimate
\[
\hat u(r,t)\geq \psi_\delta(\underline h(t-T)-r) \mbox{ for } t>T, \; r\in [0, \underline h(t-T)],
\]
we easily obtain the following conclusion:

{\it For any given small $\epsilon>0$, there exists $\delta^*>0$ small  such that for every $\delta\in (0, \delta^*]$, there exists $T^*_\delta>0$ large so that}
\[
\hat u(r,t)\geq \psi_\delta(\underline h(t-T)-r)\geq 1-\epsilon \mbox{ for } r\in [0, (c^*_{\mu_1}-\epsilon)t],\; t\geq T^*_\delta.
\]
It follows that
\[
\liminf_{t\to\infty} \Big[\min_{r\in[0, (c^*_{\mu_1}-\epsilon)t]}\hat u(r,t)\Big]\geq 1.
\]
Hence
\[
\liminf_{t\to\infty} \Big[\min_{r\in[0, (c^*_{\mu_1}-\epsilon)t]} u(r,t)\Big]\geq 1.
\]
By comparison with an ODE upper solution, it is easily seen that
\[
\limsup_{t\to\infty}u(r,t)\leq 1 \mbox{ uniformly for } r\in [0,\infty).
\]
We thus obtain, for any small $\epsilon>0$,
\[
\lim_{t\to\infty} \Big[\max_{r\in[0, (c^*_{\mu_1}-\epsilon)t]} |u(r,t)-1|\Big]= 0.
\]
This proves the second identity in \eqref{s1-u}.

Combining \eqref{estimate-1} and \eqref{lower estimate}, we obtain the first identity in \eqref{s1-u}:
\[
\lim_{t\to\infty}\frac{s_1(t)}{t}=c^*_{\mu_1}.
\]

\bigskip

\noindent{\bf Step 4:}
We complete the proof of Theorem~\ref{thm: speed} by finally showing that, for any small $\epsilon>0$,
\be\label{v-to-0}
\lim_{t\to\infty}\Big[\max_{r\in[0, (c^*_{\mu_1}-\epsilon)t]}v(r,t)\Big]=0.
\ee
We prove this by making use of \eqref{s1-u}. Suppose by way of contradiction that \eqref{v-to-0} does not hold.
Then for some $\epsilon_0>0$ small there exist  $\delta_0>0$ and a sequence $\{(r_k, t_k)\}_{k=1}^\infty$ such that
\[
\lim_{k\to\infty}t_k=\infty,\; r_k\in [0, (c^*_{\mu_1}-\epsilon_0)t_k],\; v(r_k, t_k)\geq \delta_0 \mbox{ for all } k\geq 1.
\]
By passing to a subsequence, we have either (i) $r_k\to r^*\in [0, \infty)$ or (ii) $r_k\to\infty$ as $k\to\infty$.

In case (i) we define
\[
v_k(r,t):=v(r, t+t_k),\; u_k(r,t):=u(r, t+t_k) \mbox{ for } k\geq 1.
\]
Then
\[
\partial_t v_k=\Delta v_k+v_k(1-v_k-hu_k) \mbox{ for } r\in [0, s_2(t+t_k)),\; t\geq -t_k.
\]
By \eqref{s1-u}, we have $u_k\to 1$ in $L^\infty_{loc}([0,\infty)\times \mathbb{R}^1)$.
Since $v_k(1-v_k-hu_k)$ has an $L^\infty$ bound that is independent of $k$, by standard parabolic regularity and a compactness
consideration, we may assume, by passing to a subsequence involving a diagonal process, that
\[
v_k(r,t)\to v^*(r,t) \mbox{ in } C_{loc}^{1+\alpha, \frac{1+\alpha}{2}}([0,\infty)\times \mathbb{R}^1),\; \alpha\in (0,1),
\]
and $v^*\in W^{2,1}_{p, loc}([0,\infty)\times \mathbb{R}^1)$ $(p>1)$ is a solution of
\[
\begin{cases}
v^*_t=\Delta v^*+v^*(1-h-v^*) &\mbox{ for } r\in[0,\infty),\; t\in\mathbb{R}^1,\\
 v^*_r(0,t)=0 &\mbox{ for } t\in\mathbb{R}^1.
\end{cases}
\]
Moreover, ${v^*(r^*,0)}\geq \delta_0$ and due to $\limsup_{t\to\infty} v(r,t)\leq 1$ we have $v^*(r,t)\leq 1$.

Fix $R>0$ and let $\hat v(r,t)$ be the unique solution of
\[\begin{cases}
\hat v_t=\Delta \hat v+\hat v(1-h-\hat v) &\mbox{ for } r\in[0, R),\; t>0,\\
 \hat v_r(0,t)=0,\; \hat v(R,t)=1 &\mbox{ for } t>0,\\
\hat v(r,0)=1 &\mbox{ for } r\in [0,R].
\end{cases}
\]
By the comparison principle we have, for any $s>0$,
\[
0\leq v^*(r,t)\leq \hat v(r, t+s) \mbox{ for } r\in [0, R],\; t\geq -s.
\]
On the other hand, by the well known properties of logistic type equations, we have
\[
\hat v(r,t)\to V_R(r) \mbox{ as  $t\to\infty$ uniformly for $r\in [0, R]$},
\]
where $V_R(r)$ is the unique solution to
\[
\Delta V_R+V_R(1-h-V_R) \mbox{ in } [0, R];\; V_R'(0)=0,\; V_R(R)=1.
\]
It follows that
\be\label{v*=0}
\delta_0\leq v^*(r^*, 0)\leq\lim_{s\to\infty} \hat v(r^*, s)=V_R(r^*).
\ee
By Lemma 2.1 in \cite{DM}, we have $V_R\leq V_{R'}$ in $[0, R']$ if $0<R'<R$. Hence
$V_\infty(r):=\lim_{R\to\infty} V_R(r)$ exists, and it is easily seen that $V_\infty$ is a nonnegative solution of
\[
\Delta V_\infty+V_\infty(1-h-V_\infty)=0 \mbox{ in } \mathbb{R}^1.
\]
Since $1-h<0$, by Theorem 2.1 in \cite{DM}, we have $V_\infty\equiv 0$. Hence $\lim_{R\to\infty} V_R(r)=0$ {for every $r\geq 0$}. 
We may now let $R\to\infty$ in \eqref{v*=0} to obtain
$
\delta_0\leq 0$. Thus we reach a contradiction in case (i).

In case (ii), $r_k\to\infty$ as $k\to\infty$, and we define
\[
v_k(r,t):=v(r+r_k, t+t_k),\; u_k(r,t):=u(r+r_k, t+t_k) \mbox{ for } k\geq 1.
\]
Since $r_k\leq (c^*_{\mu_1}-\epsilon_0)t_k$, by \eqref{s1-u}  we see that $u_k(r,t)\to 1$
in $L^\infty_{loc}(\mathbb{R}^1\times \mathbb{R}^1)$.
Then similarly, by passing to a subsequence, $v_k(r,t)\to \tilde v^*(r,t)$ in $C_{loc}^{1+\alpha, \frac{1+\alpha}{2}}(\mathbb{R}^1\times \mathbb{R}^1),\; \alpha\in (0,1)$,
and $\tilde v^*\in W^{2,1}_{p, loc}(\mathbb{R}^1\times \mathbb{R}^1)$ $(p>1)$ is a solution of
\[
\tilde v^*_t=\tilde v^*_{rr}+\tilde v^*(1-h-\tilde v^*) \mbox{ for } (r,t)\in\mathbb{R}^2.
\]
Moreover, $\tilde v^*(0,0)\geq \delta_0$ and $\tilde v^*(r,t)\leq 1$. We may now compare $\tilde v^*$ with the one-dimensional version of $\hat v(r, t)$ used in case (i) to obtain a contradiction. We omit the details as they are just obvious modifications of the arguments in case (i).

As we arrive at a contradiction in both cases (i) and (ii), \eqref{v-to-0} must hold. The proof is now complete.
\end{proof}

\section{Appendix}

This section is divided into three subsections. In subsection 3.1, we establish the local existence and uniqueness of solutions for a rather general system including {\bf (P)} as a special case. In subsection 3.2, we prove the global existence with some additional assumptions on the general system considered in subsection 3.1, but the resulting system is still much more general than {\bf (P)}. In the final subsection, we give the proof of Theorem 2.
\subsection{ Local existence and uniqueness}
\setcounter{equation}{0}

In this subsection, for possible future applications, we show the local existence and uniqueness of the solution to a more general system than  {\bf(P)}.
Our approach follows that in \cite{GW2} with
suitable changes, and in particular, we will fill in a  gap in the argument of \cite{GW2}.

More precisely,
we consider the following  problem:
\bea\label{General FBP}
\left\{
 \begin{array}{ll}
\vspace{2mm}
u_t=d_1\Delta u+f(r,t,u,v) \mbox{ for } 0<r<{s}_1(t),\ t>0,\\
\vspace{2mm}
v_t=d_2\Delta v+g(r,t,u,v) \mbox{ for } 0<r<{s}_2(t),\ t>0,\\
\vspace{2mm}
u_r(0,t)=v_r(0,t)=0 \mbox{ for } t>0,\\
\vspace{2mm}
u\equiv0 \mbox{ for }  r\geq {s}_1(t)\ \mbox{and}\ t>0;\;  v\equiv0 \mbox{ for }  r\geq {s}_2(t)\ \mbox{and}\ t>0,\\
\vspace{2mm}
s_{1}'(t)=-\mu_1 u_r(s_1(t),t),\;   s_{2}'(t)=-\mu_2 v_r(s_2(t),t) \mbox{ for } t>0,\\
\vspace{2mm}
({s}_1(0),{s}_2(0))=(s_1^0,s_2^0),\; (u,v)(r,0)=(u_0,v_0)(r)\ \mbox{for}\ r\in[0,\infty),
\end{array}
\right.
\eea
where $r=|x|$, $\Delta \varphi:=\varphi_{rr}+\frac{(N-1)}{r}\phi_r$, and
the initial data satisfies \eqref{ic}. We  assume that the nonlinear terms $f$ and $g$ satisfy
\[\mbox{
\bf(H1):} \hspace{1cm}\left\{\begin{array}{l} \mbox{(i) $f$ and $g$ are continuous in $r,t,u,v\in[0,\infty)$,}\smallskip\\
\mbox{(ii) $f(r,t, 0,v)=0=g(r,t, u,0)$ for $r,t, u,v\ge 0$,}\smallskip\\
\mbox{(iii) $f$ and $g$ are locally Lipschitz continuous in $r,u,v\in[0,\infty)$, }\\
\mbox{\;\;\;\;\; \; uniformly for $t$ in bounded subsets of $[0,\infty)$.}
\end{array}\right.
\]

We have the following local existence and uniqueness result for \eqref{General FBP}.
\begin{thm}\label{thm:local existence}
Assume {\bf(H1)} holds and $\alpha\in(0,1)$.
Suppose for some $M>0$,
\beaa\label{initial condition}
\|u_0\|_{C^2([0,s_1^0])}+\|v_0\|_{C^2([0,s_2^0])}+ s_1^0+s_2^0\leq M.
\eeaa
Then there exist $T\in(0,1)$ and $\widehat{M}>0$
depending only on $\alpha$, $M$ and the local Lipschitz constants of $f$ and $g$
such that  problem \eqref{General FBP}
has a unique solution
\beaa
(u,v,s_1,s_2)\in C^{1+\alpha,(1+\alpha)/2}({D^1_{T}})\times C^{1+\alpha,(1+\alpha)/2}({D^2_{T}})
\times C^{1+\alpha/2}([0,{T}]) \times C^{1+\alpha/2}([0,{T}])
\eeaa
satisfying
\bea\label{solution-estimate}
\|u\|_{C^{1+\alpha,(1+\alpha)/2}(D^1_T)}+\|v\|_{C^{1+\alpha,(1+\alpha)/2}(D^2_T)}
+\sum_{i=1}^2\|s_i\|_{C^{1+\alpha/2}([0,T])}\leq \widehat{M},
\eea
where $D^i_{T}:=\{(x,t):0\leq x \leq s_i(t),\ 0\leq t\leq T\}$ for $i=1,2$.
\end{thm}
\begin{proof}
Firstly, for given $T\in(0,1)$, we introduce the function spaces
\beaa
\Sigma_T^i:=\big\{s\in C^1([0,T]):\, s(0)=s_i^0,\ s'(0)=s^*_i,\ 0 \leq s'(t) \leq s^*_i+1,\ t\in[0,T]\big\},\quad i=1,2,
\eeaa
where
\[
\mbox{$s^*_1:=-\mu_1 u_0'(s_1^0)$,\; $s^*_2:=-\mu_2 v_0'(s_2^0)$.}
\]
Clearly $s(t)\geq s_i^0$ for $t\in[0,T]$ if $s\in\Sigma_T^i$.

For given $(\hat{s}_1,\hat{s}_2)\in\Sigma_T^1\times\Sigma_T^2$, we introduce two corresponding function spaces
\beaa
&&X^1_{T}=X_T^1(\hat s_1,\hat s_2):=\big\{u\in C([0,\infty)\times[0,T]): u\equiv0\ \mbox{for}\   r\geq \hat s_1(t),\ t\in[0,T],\\
&&\hspace{7cm} u(r,0)\equiv u_0(r),\; \|u-u_0\|_{L^\infty([0,\infty)\times[0,T])}\leq1\big\};\\
&&X^2_{T}=X_T^2(\hat s_1,\hat s_2):=\big\{v\in C([0,\infty)\times[0,T]): \ v\equiv0\ \mbox{for}\   r\geq \hat s_2(t),\ t\in[0,T],\\
&&\hspace{7cm} v(r,0)\equiv v_0(r),\;\|v-v_0\|_{L^\infty([0,\infty)\times[0,T])}\leq1\big\}.
\eeaa
We note that $X^1_T$ and $X^2_T$ are closed subsets of $C([0,\infty)\times [0,T])$ under the $L^\infty([0,\infty)\times [0,T])$ norm.

Given  $(\hat{s}_1,\hat{s}_2)\in\Sigma_T^1\times\Sigma_T^2$ and $(\hat u, \hat v)\in X_T^1\times X_T^2$, we consider the following problem
\bea\label{FBP-s hat}
\left\{
 \begin{array}{ll}
\vspace{2mm}
u_t=d_1\Delta u+f(r,t,\hat u,\hat v) \mbox{ for } 0<x<\hat{s}_1(t),\ 0<t<T,\\
\vspace{2mm}
v_t=d_2\Delta v+g(r,t,\hat u,\hat v) \mbox{ for }  0<x<\hat{s}_2(t),\ 0<t<T,\\
\vspace{2mm}
u_r(0,t)=v_r(0,t)=0 \mbox{ for } 0<t<T,\\
\vspace{2mm}
u\equiv0\;\mbox{for}\  r\geq \hat{s}_1(t)\ \mbox{and}\ t>0;\;  v\equiv0\;\mbox{for}\   r\geq \hat{s}_2(t)\ \mbox{and}\ t>0,\\
\vspace{2mm}
(\hat{s}_1,\hat{s}_2)(0)=(s_1^0,s_2^0),\ (u,v)(r,0)=(u_0,v_0)(r)\ \mbox{for}\ r\in[0,\infty).
 \end{array}
\right.
\eea

To solve \eqref{FBP-s hat} for $u$,
we straighten the boundary $r=\hat{s}_1(t)$ by the transformation $R:={r}/{\hat{s}_1(t)}$ and define
\beaa
U(R,t):=u(r,t),\quad V(R,t):=v(r,t),\quad \hat U(R,t):=\hat u(r,t),\quad \hat V(R,t):=\hat v(r,t).
\eeaa
Then $U$ satisfies
\bea\label{straighten U}
\left\{
 \begin{array}{ll}
\vspace{1mm}
\dps U_t=\frac{d_1\Delta U}{(\hat{s}_1(t))^2}+\frac{\hat{s}'_1(t)R}{\hat{s}_1(t)}U_R+\tilde f(R,t) & \mbox{ for } R\in (0,1),\, t\in (0,T),\\
\vspace{1mm}
\dps U_R(0,t)=U(1,t)=0& \mbox{ for } t\in (0,T),\\
\vspace{1mm}
\dps U(R,0)=U^0(R):=u_0(s_1^0R)& \mbox{ for } R\in[0, 1],
 \end{array}
\right.
\eea
where
\beaa\label{laplace}
\Delta U:=U_{RR}+\frac{N-1}{R}U_R,\;  \tilde f(R,t):=f(\hat s_1(t)R,t,\hat{U},\hat{V}).
\eeaa

Since
\beaa
&&s_1^0\leq \hat{s}_1(t)\leq s_1^0+s_1^*+1 \mbox{ for } t\in[0,T],  \mbox{ and }\\
&&\big\|{\hat{s}'_1}/{\hat{s}_1}\big\|_{L^{\infty}([0,T])}+\|\tilde f\|_{L^\infty([0,\infty)\times[0,T])}<\infty,
\eeaa
one can apply the standard parabolic $L^p$ theory and the Sobolev embedding theorem (see \cite{Hu,LSU})
to deduce that  \eqref{straighten U}
has a unique solution
$
U\in {C^{1+\alpha,(1+\alpha)/2}([0,1]\times[0,T])}$ with
\[
\|{U}\|_{C^{1+\alpha,(1+\alpha)/2}([0,1]\times[0,T])}\leq C_1(\|\tilde f\|_\infty+\|u_0\|_{C^2})
\]
for some $C_1$ depending only on $\alpha\in (0,1)$ and $M$. It follows that $u(r,t)=U(\frac{r}{\hat s_1(t)}, t)$
satisfies
\bea\label{u-fixed domain}
\|{u}\|_{C^{1+\alpha,(1+\alpha)/2}(D^1_T)}\leq \tilde C_1 (\|\tilde f\|_\infty+\|u_0\|_{C^2})
\eea
where $\tilde C_1$ depends only on $\alpha$ and $M$, and
\[
D^1_T:=\{(r,t): r\in [0, \hat s_1(t)),\; t\in [0,T]\}.
\]

Similarly we can solve \eqref{FBP-s hat} to find a unique $v\in C^{1+\alpha,(1+\alpha)/2}(D_T^2)$ satisfying
\bea\label{v-fixed domain}
\|{v}\|_{C^{1+\alpha,(1+\alpha)/2}(D^2_T)}\leq \tilde C_2 (\|\tilde g\|_\infty+\|v_0\|_{C^2}),
\eea
where $\tilde C_2$ depends only on $\alpha$ and $M$, and
\beaa
&&{\tilde g(R,t)=g(\hat{s}_2(t)R,t,\hat{U},\hat{V}),}\\
&&D^2_T:=\{(r,t): r\in [0, \hat s_2(t)),\; t\in [0,T]\}.
\eeaa

We now define a mapping $\mathcal{G}$ over $X^1_{T}\times X^2_{T}$ by
\[
\mathcal{G}(\hat{u},\hat{v}):=(u,v),
\]
and show that $\mathcal{G}$ has a unique fixed point in $X^1_{T}\times X^2_{T}$ as long as $T\in(0,1)$ is sufficiently small, by using the contraction mapping theorem.

For $R\in [0,1]$ and $t\in [0,T]$,
\beaa
|U(R,t)-U(R,0)|&\leq &T^{(1+\alpha)/2}\|U\|_{C^{1+\alpha, (1+\alpha)/2}([0,1]\times [0,T])}\\
&\leq& C_1T^{(1+\alpha)/2}(\|\tilde f\|_\infty+\|u_0\|_{C^2}).
\eeaa
It follows that
\beaa
\|u-u_0\|_{L^\infty([0,\infty)\times [0,T])}&=&\|U-U^0\|_{C([0,1]\times [0,T])}\\
&\leq & C_1T^{(1+\alpha)/2}(\|\tilde f\|_\infty+\|u_0\|_{C^2}).
\eeaa
Similarly,
\beaa
\|v-v_0\|_{L^\infty([0,\infty)\times [0,T])}\leq  C_2 T^{(1+\alpha)/2}(\|\tilde g\|_\infty+\|v_0\|_{C^2}).
\eeaa
This implies that $\mathcal{G}$ maps $X^1_{T}\times X^2_{T}$ into itself for small $T\in(0,1)$.

To see that $\mathcal{G}$ is a contraction mapping, we choose any $(\hat{u}_i,\hat{v}_i)\in X_T^1\times X_T^2$, $i=1,2$,
and set
\beaa
\tilde{u}:=\hat u_1-\hat u_2,\quad \tilde{v}:=\hat v_1-\hat v_2.
\eeaa
Then $(\tilde{u},\tilde{v})$ satisfies
\beaa\label{contraction-UV}
\left\{
 \begin{array}{ll}
\dps \tilde{u}_t= d_1\Delta \tilde{u}
+f(r,t,\hat{u}_1,\hat{v}_1)-f(r,t,\hat{u}_2,\hat{v}_2)\mbox{ for }  0<r<\hat s_1(t),\ 0<t<T,\\
\dps \tilde{v}_t=d_2\Delta\tilde{v}
+g(r,t,\hat{u}_1,\hat{v}_1)-g(r,t,\hat{u}_2,\hat{v}_2) \mbox{ for } 0<r<\hat s_2(t),\ 0<t<T,\\
\dps \tilde{u}_r(0,t)=\tilde{v}_r(0,t)=0 \mbox{ for } 0<t<T,\\
\dps \tilde{u}\equiv0\quad\mbox{for }\  r\geq \hat s_1(t),\ 0<t<T;\quad  \tilde{v}\equiv0\quad\mbox{for}\   r\geq \hat s_2(t),\ 0<t<T,\\
\dps  (\tilde{u},\tilde{v})(R,0)=(0,0),\ r\in[0,\infty).
 \end{array}
\right.
\eeaa
By the Lipschitz continuity of $f$ and $g$, there exists $C_0>0$ such that for $r\in [0,\max\{\hat s_1(T),\hat\sigma_1(T)\}]$ and $t\in [0, T]$,
\[
|f(r,t,\hat{u}_1,\hat{v}_1)-f(r,t,\hat{u}_2,\hat{v}_2)|\leq C_0(|\hat u_1-\hat u_2|+|\hat v_1-\hat v_2|),
\]
\[
|g(r,t,\hat{u}_1,\hat{v}_1)-g(r,t,\hat{u}_2,\hat{v}_2)|\leq C_0(|\hat u_1-\hat u_2|+|\hat v_1-\hat v_2|).
\]
We may then repeat the arguments leading to \eqref{u-fixed domain} and \eqref{v-fixed domain} to obtain
\beaa
&&\|\tilde{u}\|_{C^{1+\alpha,(1+\alpha)/2}(D^1_T)}+\|\tilde{v}\|_{C^{1+\alpha,(1+\alpha)/2}(D^2_T)}\\
 && \leq C_2(\|\hat{u}_1-\hat{u}_2\|_{L^\infty([0,\infty)\times[0,T])}+\|\hat{v}_1-\hat{v}_2\|_{L^\infty([0,\infty)\times[0,T])}),
\eeaa
for some $C_{2}=C_2(\alpha, M, C_0)$.

If we define
\[
\tilde U(R, t):=\tilde u(\hat s_1(t)R, t),\; \tilde V(R, t):=\tilde v(\hat s_2(t)R, t),
\]
then
\beaa
&&\|\tilde{U}\|_{C^{1+\alpha,(1+\alpha)/2}([0,1]\times [0,T])}+\|\tilde{V}\|_{C^{1+\alpha,(1+\alpha)/2}([0,1]\times [0,T])}\\
 && \leq C\left(
\|\tilde{u}\|_{C^{1+\alpha,(1+\alpha)/2}(D^1_T)}+\|\tilde{v}\|_{C^{1+\alpha,(1+\alpha)/2}(D^2_T)}\right)
\eeaa
for some $C=C(M)$.
Hence from the above estimate for $\tilde u$ and $\tilde v$ we obtain
\beaa
&&\|\tilde{U}\|_{C^{1+\alpha,(1+\alpha)/2}([0,1]\times [0,T])}+\|\tilde{V}\|_{C^{1+\alpha,(1+\alpha)/2}([0,1]\times [0,T])}\\
 && \leq C'_2(\|\hat{u}_1-\hat{u}_2\|_{L^\infty([0,\infty)\times[0,T])}+\|\hat{v}_1-\hat{v}_2\|_{L^\infty([0,\infty)\times[0,T])}),
\eeaa
for some $C'_{2}=C_2'(\alpha, M, C_0)$. Since $\tilde U(R,0)=\tilde V(R,0)\equiv 0$, it follows that
\beaa
&&\|\tilde{U}\|_{C([0,1]\times [0,T])}+\|\tilde{V}\|_{C([0,1]\times [0,T])}\\
 && \leq C_2' T^{(1+\alpha)/2}(\|\hat{u}_1-\hat{u}_2\|_{L^\infty([0,\infty)\times[0,T])}+\|\hat{v}_1-\hat{v}_2\|_{L^\infty([0,\infty)\times[0,T])}),
\eeaa
and hence
\beaa
&&\|\tilde{u}\|_{L^\infty([0,\infty)\times[0,T])}+\|\tilde{v}\|_{L^\infty([0,\infty)\times[0,T])}\\
 && \leq C_2'T^{(1+\alpha)/2}(\|\hat{u}_1-\hat{u}_2\|_{L^\infty([0,\infty)\times[0,T])}+\|\hat{v}_1-\hat{v}_2\|_{L^\infty([0,\infty)\times[0,T])}).
\eeaa
This implies that $\mathcal{G}$ is a contraction mapping
as long as $T\in(0,1)$ is sufficiently small.
By the contraction mapping theorem, $\mathcal{G}$ has a unique fixed point in $X_T^1\times X_T^2$,  which we denote by $(\hat u,\hat v)$.
Furthermore,  from \eqref{u-fixed domain} and  \eqref{v-fixed domain}, we have
\bea\label{hat-u-v-estimate}
\|\hat{u}\|_{C^{1+\alpha,(1+\alpha)/2}(D^1_T)}+\|\hat{v}\|_{C^{1+\alpha,(1+\alpha)/2}(D^2_T)}\leq \widehat{C}_1,
\eea
for some $\widehat{C}_1=\widehat C_1(\alpha, M, C_0)$.

For such $(\hat{u},\hat{v})$,
we introduce the mapping
\[
\mathcal{F}(\hat{s}_1,\hat{s}_2)=\mathcal{F}(\hat{s}_1,\hat{s}_2; \hat u,\hat v):=  (\bar{s}_1,\bar{s}_2)
\]
 with
\beaa
&&\bar{s}_1(t)=s_1^0-\mu_1\int_0^t {\hat{u}_r}(\hat{s}_1(\tau),\tau)d\tau,\quad t\in[0,T];\label{F-map-1}\\
&&\bar{s}_2(t)=s_2^0-\mu_2\int_0^t {\hat{v}_r}(\hat{s}_2(\tau),\tau)d\tau,\quad t\in[0,T].\label{F-map-2}
\eeaa
Clearly
\be\label{bar-s1-s2}
\bar s_1'(t)=-\mu_1\hat u_r(\hat s_1(t),t)\geq 0,\; \bar s_2'(t)=-\mu_2\hat v_r(\hat s_2(t),t)\geq 0 \mbox{ for } t\in [0,T].
\ee

We shall again apply the contraction mapping theorem to deduce that $\mathcal{F}$ defined on $\Sigma_T^1\times \Sigma_T^2$ has a unique fixed point.
By \eqref{hat-u-v-estimate} and \eqref{bar-s1-s2}, we see that $\bar{s}'_i\in C^{{\alpha}/{2}}([0,T])$ with
\bea\label{estimate-moving bdry}
\sum_{i=1}^{2}\|\bar{s}'_i\|_{C^{{\alpha}/{2}}([0,T])}\leq (\mu_1+\mu_2)\widehat{C}_1.
\eea
It follows that
\beaa
\sum_{i=1}^{2}\|\bar{s}'_i-s^*_i\|_{C([0,T])}\leq (\mu_1+\mu_2)\widehat{C}_1T^{{\alpha}/{2}}.
\eeaa
Hence $\mathcal{F}$ maps $\Sigma_T^1\times \Sigma_T^2$ into itself as long as $T\in(0,1)$ is sufficiently small.

To show that $\mathcal{F}$ is a contraction mapping, we let $(\hat{u}^{s},\hat{v}^{s})$ and $(\hat{u}^{\sigma},\hat{v}^{\sigma})$
be two fixed points of $\mathcal{G}$
associated with  $(\hat{s}_1,\hat{s}_2)$ and $(\hat{\sigma}_1,\hat{\sigma}_2)\in\Sigma_T^1\times\Sigma_T^2$,
respectively; and for $i=1,2$, we denote $D_T^i$ associated to  $(\hat{s}_1,\hat{s}_2)$ and $(\hat{\sigma}_1,\hat{\sigma}_2)$ by, respectively
\[
D_{T,s}^i \mbox{ and } D_{T,\sigma}^i.
\]

Let us straighten $r=\hat{s}_1(t)$ and $r=\hat{\sigma}_1(t)$, respectively. To do so for $r=\hat s_1(t)$, we define
\beaa
U^{s}(R,t):=\hat{u}^{s}(r,t), \quad V^{s}(R,t):=\hat{v}^{s}(r,t),\quad R=\frac{r}{\hat{s}_1(t)};
\eeaa
then $U^{s}$ satisfies
\bea\label{straighten Us}
\left\{
 \begin{array}{ll}
\vspace{1mm}
\dps U^s_t=\frac{d_1\Delta U^s}{(\hat{s}_1(t))^2}+\frac{\hat{s}'_1(t)R}{\hat{s}_1(t)}U^s_R+\tilde f^s(R,t) & \mbox{ for } R\in (0,1),\, t\in (0,T),\\
\vspace{1mm}
\dps U^s_R(0,t)=U^s(1,t)=0 &\mbox{ for } t\in (0,T),\\
\vspace{1mm}
\dps U^s(R,0)=u_0(s_1^0R)& \mbox{ for } R\in[0, 1],
 \end{array}
\right.
\eea
where
\[
\tilde f^s(R,t):=f(\hat s_1(t)R,t, U^s, V^s).
\]

Similarly we set
\beaa
U^{\sigma}(R,t):=\hat{u}^{\sigma}(r,t), \quad V^{\sigma}(R,t):=\hat{v}^{\sigma}(r,t),\quad R=\frac{r}{\hat{\sigma}_1(t)},
\eeaa
and find that \eqref{straighten Us} holds with $(U^s, V^s, \hat{s}_1(t))$  replaced by
$(U^\sigma, V^\sigma,\hat{\sigma}_1(t))$ everywhere.

Next we introduce
\beaa
&& \eta(t):=\hat s_1(t)/\hat s_2(t),\; \xi(t):=\hat\sigma_1(t)/\hat\sigma_2(t),\\
&&P(R,t):=U^{s}(R,t)-U^{\sigma}(R,t),\quad Q(R,t):=V^{s}(R,t)-V^{\sigma}(R,t).
\eeaa
By some simple computations, $P$ satisfies
\be\label{P}
\left\{
 \begin{array}{ll}
 \vspace{2mm}
\dps P_t=\frac{d_1\Delta P}{(\hat{s}_1(t))^2}+\frac{\hat{s}'_1(t)R P_{R}}{\hat{s}_1(t)}+d_1B_1(t)U^{\sigma}_{RR}+R B_2(t)U^{\sigma}_{R}
+F(R,t)\vspace{-3mm}\\
\hspace{7cm} \mbox{ for } R\in [0,1],\, t\in[0,T],\\
\vspace{2mm}
\dps P_R(0,t)=P(1,t)=0 \mbox{ for } t\in[0,T],\\
\vspace{2mm}
\dps P(R,0)=0 \mbox{ for } R\in[0, 1],
 \end{array}
\right.
\ee
where
\beaa
&&B_1(t):=\frac{1}{(\hat{s}_1(t))^2}-\frac{1}{(\hat{\sigma}_1(t))^2},\quad B_2(t):=\frac{\hat{s}_1'(t)}{\hat{s}_1(t)}-\frac{\hat{\sigma}_1'(t)}{\hat{\sigma}_1(t)},\\
&&F(R,t):=f(R\hat{s}_1(t),t, U^{s},V^{s})-f(R\hat{\sigma}_1(t),t, U^\sigma,V^\sigma).
\eeaa

In view of \eqref{bar-s1-s2},
\beaa
\bar{s}'_1(t)=-\mu_1\frac{U^s_R(1,t)}{\hat{s}_1(t)},\quad \bar{\sigma}'_1(t)=-\mu_1\frac{U^{\sigma}_R(1,t)}{\hat{\sigma}_1(t)},
\eeaa
and hence
\beaa
\bar{s}'_1(t)-\bar{\sigma}'_1(t)
=\frac{\mu_1}{\hat s_1(t)}\big[U_R^\sigma(1,t)-U_R^s(1,t)\big] +\frac{\mu_1 U_R^\sigma(1,t)}{\hat s_1(t)\hat\sigma_1(t)}
\big[\hat s_1(t)-\hat \sigma_1(t)\big].
\eeaa

From now on, we will depart from the approach of \cite{GW2} and fill in a gap which occurs in the argument there towards the proof
 that $\mathcal{F}$ is a contraction mapping.

It follows from the above identity that
\beaa
\|\bar s'_1-\bar\sigma'_1\|_{C^{\frac{1+\alpha}{2}}([0,T])}\leq C\Big(\|U_R^s(1,\cdot)-U_R^\sigma(1,\cdot)\|_{C^{\frac{1+\alpha}{2}}([0,T])}+
\|\hat s_1-\hat\sigma_1\|_{C^{\frac{1+\alpha}{2}}([0,T])}\Big),
\eeaa
where $C$ depends on $\mu_1$ and the upper bounds of $\|\hat s_1\|_{C^{(1+\alpha)/2}([0,T])}$, $\|\hat\sigma_1\|_{C^{(1+\alpha)/2}([0,T])}$
and $\|U_R^\sigma(1,\cdot)\|_{C^{(1+\alpha)/2}([0,T])}$. Hence $C=C(\alpha, M,C_0)$.

Since $T\leq 1$, clearly
\[
\|\hat s_1-\hat\sigma_1\|_{C^{\frac{1+\alpha}{2}}([0,T])}\leq \|\hat s_1'-\hat\sigma_1'\|_{C([0,T])}.
\]
We also have
\[
\|U_R^s(1,\cdot)-U_R^\sigma(1,\cdot)\|_{C^{\frac{1+\alpha}{2}}([0,T])}\leq \|P\|_{C^{1+\alpha,(1+\alpha)/2}([0,1]\times [0,T])}.
\]
We thus obtain
\be\label{s-sigma}
\|\bar s'_1-\bar\sigma'_1\|_{C^{\frac{1+\alpha}{2}}([0,T])}\leq C \Big(\|P\|_{C^{1+\alpha,(1+\alpha)/2}([0,1]\times [0,T])}+\|\hat s_1'-\hat\sigma_1'\|_{C([0,T])}\Big).
\ee

Applying the $L^p$ estimate and the Sobolev embedding theorem to the problem \eqref{P}, we obtain, for some $p>1$,
\[
\begin{array}{l}
\|P\|_{C^{1+\alpha, (1+\alpha)/2}([0,1]\times[0,T])}\smallskip \\
\leq M_4\Big(\|B_1\|_{C([0,T])}\|U_{RR}^\sigma\|_{L^p([0,1]\times [0,T])}\smallskip\\
\hspace{2cm}+\|B_2\|_{C([0,T])}\|U_{R}^\sigma\|_{L^p([0,1]\times [0,T])}+\|F\|_{L^p([0,1]\times[0,T])} \Big)
\end{array}
\]
for some $M_4>0$ depending only on $\alpha$ and $M$. Due to the $W^{2,1}_p([0,1]\times [0,T])$ bound for $U^\sigma$, we
hence obtain
\be\label{P-bd}
\begin{array}{l}
\|P\|_{C^{1+\alpha, (1+\alpha)/2}([0,1]\times[0,T])}\smallskip \\
\leq M_5\Big(\|B_1\|_{C([0,T])}+\|B_2\|_{C([0,T])}+\|F\|_{L^p([0,1]\times[0,T])} \Big)
\end{array}
\ee
for some $M_5>0$ depending only on $\alpha$, $M$ and the Lipschitz constant of $f$.

By the definitions of $B_1(t),\; B_2(t)$ and $F(R,t)$, we have
\be\label{B1-bd}
\|B_1\|_{C([0,T])}\leq C\|\hat s_1-\hat \sigma_1\|_{C([0,T])},
\ee
\be\label{B2-bd}
 \|B_2\|_{C([0,T])}\leq C\Big(\|\hat s_1-\hat \sigma_1\|_{C([0,T])}
+\|\hat s_1'-\hat \sigma_1'\|_{C([0,T])}\Big),
\ee
and
\be\label{F-bd}
\|F\|_{L^p([0,1]\times [0,T])}\leq C\Big(
\|P\|_{C([0,1]\times[0,T])}+\|Q\|_{C([0, 1]\times[0,T])}+\|\hat{s}_1-\hat{\sigma_1}\|_{C[0,T]}\Big),
\ee
for some $C>0$ depending only on $M$ and the Lipschitz constants of $f$.

We next estimate $\|P\|_{C([0,1]\times[0,T])}$ and $\|Q\|_{C([0, 1]\times[0,T])}$ by using the estimate in
Lemma 2.2 of \cite{GW2}, namely
\[
\|\hat u^s-\hat u^\sigma\|_{C(\Gamma_T^1)}+\|\hat v^s-\hat v^\sigma\|_{C(\Gamma_T^2)}\leq C\sum_{i=1}^2\|\hat s_i-\hat\sigma_i\|_{C([0,T])},
\]
where
\[\Gamma_T^i:=\big\{(r,t): 0\leq r\leq\min\{\hat s_i(t),\hat\sigma_i(t)\}\big\},\; i=1,2.
\]

Without loss of generality, we may assume $\hat s_1(t)\leq \hat\sigma_1(t)$. Then for any $R\in [0,1]$ and $t\in [0,T]$,
we have
\beaa
|P(R,t)|&=&|\hat u^s(R\hat s_1(t),t)-\hat u^\sigma(R\hat\sigma_1(t), t)|\\
&\leq& |\hat u^s(R\hat s_1(t), t)-\hat u^\sigma (R \hat s_1(t),t)|+
|\hat u^\sigma(R\hat s_1(t), t)-\hat u^\sigma (R \hat \sigma_1(t),t)|\\
&\leq& C\sum_{i=1}^2\|\hat s_i-\hat\sigma_i\|_{C([0,T])}+\|\hat u^\sigma\|_{C^{1+\alpha,(1+\alpha)/2}({D_{T,\sigma}^1})}\|\hat s_1-\hat \sigma_1\|_{C([0,T])}\\
&\leq &\tilde C\sum_{i=1}^2\|\hat s_i-\hat\sigma_i\|_{C([0,T])}.
\eeaa
It follows that
\[
\|P\|_{C([0,1]\times[0,T])}\leq \tilde C\sum_{i=1}^2\|\hat s_i-\hat\sigma_i\|_{C([0,T])}.
\]

For any $R\in [0,1]$ and $t\in [0,T]$,
we have
\beaa
|Q(R,t)|&=&|\hat v^s(R\hat s_1(t),t)-\hat v^\sigma(R\hat\sigma_1(t), t)|\\
&=& |\hat v^s(R\eta(t)\hat s_2(t), t)-\hat v^\sigma (R\xi(t) \hat \sigma_2(t),t)|.
\eeaa

We now consider all the possible cases:

(i)
If $R\eta(t)\geq 1$ and $R\xi(t)\geq 1$, then we immediately obtain
\[
|Q(R,t)|=0.
\]

(ii) If $R\eta(t)<1$ and $R\xi(t)<1$, assuming without loss of generality $\hat s_2(t)\leq \hat \sigma_2(t)$, then
\beaa
|Q(R,t)|
&=& |\hat v^s(R\eta(t)\hat s_2(t), t)-\hat v^\sigma (R\xi(t) \hat \sigma_2(t),t)|\\
&\leq& |\hat v^s(R\eta(t)\hat s_2(t),t)-\hat v^\sigma(R\eta(t) \hat s_2(t),t)|\\
&&\; +|\hat v^\sigma(R\eta(t)\hat s_2(t), t)-\hat v^\sigma (R\xi(t) \hat \sigma_2(t),t)|\\
&\leq&C\sum_{i=1}^2\|\hat s_i-\hat\sigma_i\|_{C([0,T])}\\
&&\; +\|\hat v^\sigma\|_{C^{1+\alpha,(1+\alpha)/2}({D_{T,\sigma}^2})}\|\hat s_1-\hat \sigma_1\|_{C([0,T])}\\
&\leq &\tilde C\sum_{i=1}^2\|\hat s_i-\hat\sigma_i\|_{C([0,T])}.
\eeaa

(iii) If $R\eta(t)<1\leq R\xi(t)$ and $R\eta(t)\hat s_2(t)\leq\hat\sigma_2(t)$, then
\beaa
|Q(R,t)|
&=& |\hat v^s(R\eta(t)\hat s_2(t), t)-\hat v^\sigma (R\xi(t) \hat \sigma_2(t),t)|\\
&=&|\hat v^s(R\eta(t) \hat s_2(t),t)|\\
&\leq& |\hat v^s(R\eta(t)\hat s_2(t), t)-\hat v^\sigma (R\eta(t)\hat s_2(t),t)|\\
&&\; +|\hat v^\sigma (R\eta(t)\hat s_2(t),t)-\hat v^\sigma(\hat\sigma_2(t),t)|\\
&\leq&C\sum_{i=1}^2\|\hat s_i-\hat\sigma_i\|_{C([0,T])}\\
&&\;+\|\hat v^\sigma\|_{C^{1+\alpha,(1+\alpha)/2}({D_{T,\sigma}^2})}|R\eta(t)\hat s_2(t)-\hat\sigma_2(t)|.
\eeaa

From $R\eta(t)<1\leq R\xi(t)$ and $R\eta(t)\hat s_2(t)\leq\hat\sigma_2(t)$ we obtain
\beaa
|R\eta(t)\hat s_2(t)-\hat\sigma_2(t)|&\leq& \left|\frac{\eta(t)}{\xi(t)}\hat s_2(t)-\hat\sigma_2(t)\right|\\
&\leq& \|1/\xi\|_{C([0,T])}|\eta(t)\hat s_2(t)-\xi(t)\hat\sigma_2(t)|\\
&=&\|1/\xi\|_{C([0,T])}|\hat s_1(t)-\hat\sigma_1(t)|.
\eeaa

Thus in this case we also have
\[
|Q(R,t)|\leq \tilde C\sum_{i=1}^2\|\hat s_i-\hat\sigma_i\|_{C([0,T])}.
\]

(iv) If $R\eta(t)<1\leq R\xi(t)$ and $R\eta(t)\hat s_2(t)>\hat\sigma_2(t)$, then
\[
|R\eta(t)\hat s_2(t)-\hat s_2(t)|\leq |\hat \sigma_2(t)-\hat s_2(t)| \mbox{ and }
\]
\beaa
|Q(R,t)|
&=&|\hat v^s(R\eta(t) \hat s_2(t),t)|\\
&=& |\hat v^s(R\eta(t)\hat s_2(t), t)-\hat v^s (s_2(t),t)|\\
&\leq&\|\hat v^s\|_{C^{1+\alpha,(1+\alpha)/2}({D_{T,s}^2})}|R\eta(t)\hat s_2(t)-\hat s_2(t)|\\
&\leq&\|\hat v^s\|_{C^{1+\alpha,(1+\alpha)/2}({D_{T,s}^2})}|\hat s_2(t)-\hat \sigma_2(t)|\\
&\leq& \tilde C\sum_{i=1}^2\|\hat s_i-\hat\sigma_i\|_{C([0,T])}.
\eeaa

(v) If $R\eta(t)\geq 1> R\xi(t)$, we are in a symmetric situation to cases (iii) and (iv) above, so we
similarly obtain
\be\label{Q-bd}
|Q(R,t)|\leq \tilde C\sum_{i=1}^2\|\hat s_i-\hat\sigma_i\|_{C([0,T])}.
\ee

Thus in all the possible cases \eqref{Q-bd} always holds. It follows that
\[
\|Q\|_{C([0,1]\times [0,T])}\leq \tilde C\sum_{i=1}^2\|\hat s_i-\hat\sigma_i\|_{C([0,T])}.
\]

We thus obtain from \eqref{F-bd} that
\be\label{F-bd2}
\|F\|_{L^p([0,1]\times [0,T])}\leq \tilde C\sum_{i=1}^2\|\hat s_i-\hat\sigma_i\|_{C([0,T])}.
\ee

We may now substitute \eqref{B1-bd}, \eqref{B2-bd} and \eqref{F-bd2} into \eqref{P-bd} to obtain
\[
\|P\|_{C^{1+\alpha, (1+\alpha)/2}([0,1]\times[0,T])}\leq \tilde C\left(\|\hat s_1'-\hat\sigma_1'\|_{C([0,T])}+
\sum_{i=1}^2\|\hat s_i-\hat\sigma_i\|_{C([0,T])}\right).
\]
It thus follows from \eqref{s-sigma} that
\[
\|\bar s'_1-\bar\sigma'_1\|_{C^{\frac{1+\alpha}{2}}([0,T])}\leq \tilde C' \left(\|\hat s_1'-\hat\sigma_1'\|_{C([0,T])}+
\sum_{i=1}^2\|\hat s_i-\hat\sigma_i\|_{C([0,T])}\right).
\]
Since $\bar s'_1(0)-\bar\sigma_1'(0)=0$, this implies
\[
\|\bar s'_1-\bar\sigma'_1\|_{C([0,T])}\leq T^{\frac{1+\alpha}{2}}\tilde C' \left(\|\hat s_1'-\hat\sigma_1'\|_{C([0,T])}+
\sum_{i=1}^2\|\hat s_i-\hat\sigma_i\|_{C([0,T])}\right).
\]
Hence for $T>0$ sufficiently small we have
\be\label{s1-sigma1}
\|\bar s'_1-\bar\sigma'_1\|_{C([0,T])}\leq T^{\frac{1+\alpha}{2}}\hat C_1
\sum_{i=1}^2\|\hat s_i-\hat\sigma_i\|_{C([0,T])}
\ee
with $\hat C_1>0$ depending only on $\alpha, M$ and the Lipschitz constant of $f$.

In a similar manner, we can straighten $r=\hat s_2(t)$ and $r=\hat\sigma_2(t)$ to obtain
\be\label{s2-sigma2}
\|\bar s'_2-\bar\sigma'_2\|_{C([0,T])}\leq T^{\frac{1+\alpha}{2}}\hat C_2
\sum_{i=1}^2\|\hat s_i-\hat\sigma_i\|_{C([0,T])}
\ee
with $\hat C_2>0$ depending only on $\alpha, M$ and the Lipschitz constant of $g$.

Finally, using $\hat{s}_i(0)=\hat{\sigma}_i(0)=s_i^0$, $i=1,2$, we see that
\bea\label{bound by derivative}
\|\hat{s}_i-\hat{\sigma}_i\|_{C([0,T])}\leq T \|\hat{s}'_i-\hat{\sigma}'_i\|_{C([0,T])}, \ i=1,2.
\eea

Combining \eqref{s1-sigma1}, \eqref{s2-sigma2} and \eqref{bound by derivative}, we see that $\mathcal{F}$ is a contraction mapping as long as $T>0$ is sufficiently small. Hence $\mathcal{F}$ has a unique fixed point $(s_1,s_2)\in \Sigma_T^1\times \Sigma_T^2$ for such $T$.

Let $(u,v)$ be the unique fixed point of $\mathcal{G}$ in $X_T^1(s_1, s_2)\times X_T^2(s_1, s_2)$; then it is easily seen that $(u,v,s_1,s_2)$ is the unique solution
of \eqref{General FBP}.
Furthermore, \eqref{solution-estimate} holds because of \eqref{hat-u-v-estimate} and \eqref{estimate-moving bdry}.
We have now completed the proof of Theorem~\ref{thm:local existence}.
\end{proof}

By Theorem~\ref{thm:local existence} and the Schauder estimate, we see that the solution of {\bf(P)} defined for $t\in[0,T]$
is actually a classical solution.

\subsection{Global existence}

In this subsection, we show that the unique local solution of \eqref{General FBP} can be extended to all positive time if the following extra assumption is imposed:
\begin{itemize}
\item[{\bf(H2)}] There exists a positive constant $K$ such that $f(r,t,u,v)\leq K(u+v)$ and $g(r,t,u,v)\leq K(u+v)$ for $r,t, u,v\ge 0$.
\end{itemize}

\begin{thm}\label{thm:global existence}
Under the assumptions of Theorem~\ref{thm:local existence} and {\bf(H2)}, problem \eqref{General FBP} has a unique globally in time solution.
\end{thm}
\begin{proof} The proof is similar to that of \cite[Theorem 2.4]{DL2}. For the reader's convenience, we present a brief proof.
Let $[0,T^*)$ be the largest time interval for which the unique solution of \eqref{General FBP} exists. By Theorem~\ref{thm:local existence}, $T_*>0$.
By the strong maximum principle, we see that $u(r,t)>0$ in $[0,s_1(t))\times[0,T_*)$ and $v(r,t)>0$ in $[0,s_2(t))\times[0,T_*)$.
We will show that $T_*=\infty$. Aiming for a contradiction, we assume that $T_*<\infty$. Consider the following ODEs
\beaa
&&dU/dt=K(U+V),\quad t>0,\quad U(0)=\|u_0\|_{L^{\infty}([0,s_1^0])},\\
&&dV/dt=K(U+V),\quad t>0,\quad V(0)=\|v_0\|_{L^{\infty}([0,s_2^0])}.
\eeaa
Take $M^*>T_*$. Clearly,
\beaa
0<U(t)+V(t)<\Big(\|u_0\|_{L^{\infty}([0,s_1^0])}+\|v_0\|_{L^{\infty}([0,s_2^0])}\Big)e^{2KM^*}:=C_1,\quad t\in[0,T_*)
\eeaa
By {\bf(H2)}, we can compare $(u,v)$ with $(U,V)$ to obtain
\beaa
\|u\|_{L^{\infty}([0,s_1(t)]\times[0,T_*))}+\|v\|_{L^{\infty}([0,s_2(t)]\times[0,T_*))}\leq C_1.
\eeaa

Next, we can use a similar argument as in \cite[Lemma 2.2]{DL} to derive
\beaa
0<s_i'(t)\leq C_2,\quad t\in(0,T_*),\ i=1,2
\eeaa
for some $C_2$ independent of $T_*$. Furthermore, we have
\beaa
s_i^0\leq s_i(t)\leq s_i^0+C_2t\leq s_i^0+C_2M^*,\quad t\in[0,T_*),\ i=1,2.
\eeaa
Taking $\epsilon\in(0,T_*)$, by standard parabolic regularity, there exists $C_3>0$ depending only on $K$, $M^*$, $C_1$ and $C_2$ such that
\beaa
\|u(\cdot,t)\|_{C^2([0,s_1(t)])}+\|v(\cdot,t)\|_{C^2([0,s_2(t)])}\leq C_3,\quad t\in[\epsilon,T_*).
\eeaa
By Theorem~\ref{thm:local existence}, there exists $\tau>0$ depending only on $K$, $M^*$ and $C_i$ ($i=1,2,3$) such that the
solution of problem \eqref{General FBP} with initial time $T_*-\tau/2$ can be extended uniquely to
the time $T_*+\tau/2$, which contradicts the definition of $T_*$. This completes the proof of Theorem~\ref{thm:global existence}.
\end{proof}
}

\subsection{Proof of Theorem~\ref{prop-trichotomy}}
\begin{proof}[Proof of Theorem~\ref{prop-trichotomy}]
Define
\beaa
s_*=R^*\sqrt{\frac{d}{r}},\quad s^*=R^*\sqrt{\frac{d}{r}}\frac{1}{\sqrt{1-k}},\quad s^{**}=R^*.
\eeaa
First, following the same lines in \cite[Theorem 2]{GW2} with some minor changes, we can prove the following three results:
\begin{itemize}
\item[(i)] If $s_{1,\infty}\leq s_*$, then $u$ vanishes eventually. In this case, $v$ spreads successfully (resp. vanishes eventually) if $s_{2,\infty}>s^{**}$ (resp. $s_{2,\infty}\leq s^{**}$),
 \item[(ii)] If $s_*<s_{1,\infty}\leq s^*$, then $u$ vanishes eventually, and $v$ spreads successfully.
 \item[(iii)] If $s_{1,\infty}> s^*$, then  $u$ spreads successfully.
\end{itemize}

\medskip

Next, we shall show
\begin{itemize}
\item[(iv)] If $s_{1,\infty}> s^*$ and $(\mu_1,\mu_2)\in \mathcal{B}$, then  $u$ spreads successfully and $v$ vanishes eventually.
\end{itemize}
By a simple comparison consideration we see that
\bea\label{prop1-upper est}
\limsup_{t\to\infty}\frac{s_2(t)}{t}\leq s^*_{\mu_2}.
\eea
By (iii), we see that $u$ spreads successfully and so $s_{1,\infty}=\infty$.
It follows that there exists $T\gg1$ such that
\beaa
s_1(T)\geq R^*\sqrt{\frac{d}{r}}\frac{1}{\sqrt{1-k}}.
\eeaa
This allows us to use a similar argument to that leading to \eqref{estimate-1} but taking $T$ as the initial time to
obtain
\bea\label{prop1-lower est}
\liminf_{t\to\infty}\frac{s_1(t)}{t}\geq c^*_{\mu_1}.
\eea

To show that $s_{2,\infty}<\infty$ we argue by contradiction and assume $s_{2,\infty}=\infty$. Since $(\mu_1,\mu_2)\in\mathcal{B}$,
from \eqref{prop1-lower est} and \eqref{prop1-upper est} we can find $\tau\gg1$ and $\hat{c}$ such that
\beaa
s^*_{\mu_2}<\hat{c}<c^*_{\mu_1},\quad s_2(t)<\hat{c}t<s_1(t)\quad\mbox{for all $t\geq\tau$}.
\eeaa
Then by the same process used in deriving the second identity in \eqref{s1-u}, we have
\be\label{prop1-u-to-1}
\lim_{t\to\infty} \Big[\max_{r\in[0, \hat{c}t]} |u(r,t)-1|\Big]= 0.
\ee
Also, noting $h>1$ and $s_2(t)<\hat{c}t$, there exist $\hat{\tau}> \tau$ such that
\beaa
v_t=\Delta v+v(1-v-hu)\leq \Delta v,\quad 0<r<s_2(t),\ t\geq \hat{\tau},
\eeaa
which leads to $s_{2,\infty}<\infty$ by simple comparison (cf. \cite[Theorem 3]{GW2}).
This reaches a contradiction. Hence we have proved $s_{2,\infty}<\infty$. Finally,
using $s_{2,\infty}<\infty$ we can show
$\lim_{t\to\infty}\|v(\cdot,t)\|_{C([0,s_2(t)])}=0$ (cf. \cite[Lemma 3.4]{GW2}). Hence
$v$ vanishes eventually and then (iv) follows.

 The conclusions of  Theorem~\ref{prop-trichotomy} follow easily from (i)-(iv).
\end{proof}

\bigskip

\noindent{\bf Acknowledgments.}
The authors are grateful to the referee for  valuable
suggestions on improving the presentation of the paper.
YD was supported by the Australian Research Council and
CHW was partially supported by the Ministry of Science and Technology of Taiwan and National Center for Theoretical Science (NCTS).
This research was initiated during the visit of CHW to the University of New England, and he is grateful for the hospitality.


\end{document}